\documentclass[11pt,oneside]{article}

\usepackage[colon,numbers]{natbib} 

\usepackage{graphicx}
\usepackage{amsmath}
\usepackage{amsfonts}
\usepackage{amssymb}
\usepackage{mathrsfs}
\usepackage{amsthm}
\usepackage{enumerate}
\usepackage{pdfsync}
\usepackage{comment}
\usepackage{thmtools}
\usepackage{bm}
\usepackage{times}
\usepackage{amsmath, amssymb, graphicx, hyperref}

    \numberwithin{equation}{section}
    
    \newcommand{\R}{\mathbb{R}}
    \newcommand{\I}{\mathcal{I}}
        \newcommand{\J}{\mathcal{J}}

\newcommand{\N}{\mathbb{N}}

\newcommand{\1}{\mathbf{1}}
\newcommand{\Prob}{\mathbb{P}}

\newcommand{\ra}{\rightarrow}

\newcommand{\dx}{\mathrm{d}}

\bmdefine\mub{\mu}
\bmdefine\etab{\eta}
\bmdefine\varthetab{\vartheta}
\bmdefine\varphib{\varphi}
\bmdefine\betab{\beta}
\bmdefine\lab{\lambda}
\bmdefine\sigmab{\sigma}
\bmdefine\varsigmab{\varsigma}
\bmdefine\taub{\tau}
\bmdefine\rhob{\rho}
\bmdefine\pib{\pi}
\bmdefine\varrhob{\varrho}
\bmdefine\gammab{\gamma}
\bmdefine\Gammab{\Gamma}

 \theoremstyle{plain}
\newtheorem{thm}{Theorem}[section]

\newtheorem{cor}[thm]{Corollary}

\theoremstyle{definition}
\newtheorem{defn}[thm]{Definition}

\theoremstyle{remark}
\newtheorem{rem}[thm]{Remark}

\newtheorem{example}[thm]{Example}

\title{{PATHWISE SOLVABILITY OF   STOCHASTIC INTEGRAL 
EQUATIONS WITH GENERALIZED DRIFT   AND NON-SMOOTH DISPERSION  FUNCTIONS
\footnote{We are deeply indebted to Hans-J\"urgen Engelbert, Cristina Di Girolami, Gechun Liang, Nicolas Perkowski, Vilmos Prokaj,  Francesco Russo, Phillip Whitman, and Marc Yor  for   very helpful comments and suggestions on the subject matter of this paper. We are grateful to the associate editor and the referee  for their careful reading of the manuscript and their insightful comments.  I.K.~acknowledges generous support from  the National Science Foundation, under  grant NSF-DMS-14-05210.  J.R.~acknowledges generous support from the Oxford-Man Institute of Quantitative Finance, University of Oxford, where a major part of this work was completed. } 
}}

\author{  
 IOANNIS KARATZAS \thanks{
Department of Mathematics,  Columbia University, New York, NY 10027, USA  (E-mail: {\it ik@math.columbia.edu}) and       \textsc{Intech} Investment Management,  One Palmer Square, Suite 441, Princeton, NJ 08542, USA    (E-mail:    {\it ik@enhanced.com}).}  
 \and
 JOHANNES RUF      \thanks{ 
Department of Mathematics, University College London, Gower Street, London WC1E 6BT, United Kingdom (E-mail:    {\it j.ruf@ucl.ac.uk }).
          }
                                      }

\begin{document}

\maketitle

\begin{abstract}
\noindent
We study 
one-dimensional stochastic integral equations with non-smooth dispersion co\"efficients, and with drift components that are not restricted to be absolutely continuous with respect to   Lebesgue measure. In the spirit of Lamperti, Doss and Sussmann, we relate solutions of such equations to solutions  of certain ordinary integral equations, indexed by a generic element of the underlying probability space. This relation  allows us to solve the stochastic integral equations in a pathwise sense.

\medskip
\noindent
{\it R\'esum\'e:} Nous \'etudions des \'equations int\'egrales stochastiques  unidimensionnelles
avec coefficient de diffusion non-r\'egulier, et avec terme de d\'erive non
 n\'ecessairement absolument continues par rapport \`a la mesure de Lebesgue. En s'inspirant de Lamperti, Doss et Sussmann,  la r\'esolution de ces \'equations  se ram\`ene \`a la r\'esolution de certaines \'equations int\'egrales ordinaires,  param\'etr\'ees par un \'el\'ement $\,\omega$ variant dans l'espace de probabilit\'e de base. Ce lien nous permet de r\'esoudre les \'equations int\'egrales stochastiques d'une fa\c{c}on  ``trajectorielle".  
\end{abstract}

\begin{center}
{ {\it Annales de l'Institut Henri Poincar\'e,} to appear.}
\end{center}



\section{Introduction}
 \label{sec0}

Stochastic integral equations (SIEs) are   powerful tools for modeling dynamical  systems subject    to random perturbations.  Any such equation has two components: a stochastic integral with respect to a process that models the ``underlying noise'' of the system, and a drift term that models some ``trend.''  In many applications, the drift term is assumed to be absolutely continuous with respect to   Lebesgue measure on the real line. However, motivated by the pioneering work of \citet{Walsh} and \citet{HarrisonShepp}  on the   ``skew Brownian motion,''  several authors have   studied SIEs without such a continuity assumption  in quite some generality,  beginning with \citet{StroockYor1981}, \citet{LeGall1984}, \citet{BarlowPerkins}, and \citet{Engelbert_Schmidt_lecture}. In the  years since,
\citet{Engelbert_Schmidt_1991}, \citet{Engelbert1991}, \citet{FlandoliRussoWolf1, FlandoliRussoWolf2}, \citet{BassChen}, \citet{RussoTrutnau},  and \citet{BleiEngelbert2012, BleiEngelbert2013} have provided deep existence and uniqueness  results about such equations.

In this present work  we extend the {\it pathwise approach} taken by \citet{Lamperti1964}, 
\citet{Doss1977} and  \citet{Sussmann1978}, who focused on the case of an absolutely continuous drift and of a smooth ($ C^2$) dispersion function, to one-dimensional SIEs with   generalized drifts and with strictly positive dispersions which, together with their reciprocals, are of  finite first variation  on compact subsets of the state space.  This pathwise approach   proceeds via a suitable transformation of the underlying SIE which replaces the stochastic integral component by the   process that models the driving noise (a Brownian motion or, more generally, a continuous semimartingale); this noise process   enters the transformed equation, now an {\it  Ordinary} (that is, non-stochastic) {\it Integral Equation}  (OIE), only parametrically through its co\"efficients.  Such a transformation emphasizes the pathwise character of the SIE, that is,   highlights the representation of the solution process (``output") as a measurable and non-anticipative functional of the driving noise (``input"). The pathwise point of view allows the modeler, who tries to solve the SIE, to construct an  input-output map  without having to worry about stochastic integration, which notoriously obscures the dependence of the solution path  on the Brownian (or semimartingale) path, due to the ``$\mathbb{L}^2$--smearing'' of stochastic integration.

We emphasize here also   the reverse implication:  If one can show, say  via probabilistic methods, that a certain SIE has a solution, then this directly yields existence results for certain OIEs. Such OIEs often may have very irregular input functions, so that such existence results would be very hard to obtain via standard analytical arguments.

\medskip
\noindent
{\bf Overview:} Section~\ref{SS:setup} provides the setup, and Section~\ref{sec02} links an SIE with generalized drift  to a collection of related OIEs.  While we rely on some rather weak assumptions on the dispersion function, such as   time-homogeneity and finite variation on compact subsets of the state space, we  make hardly any assumptions on the drift function; we allow it, for example,  to  depend explicitly on the input noise. We discuss also the Stratonovich version of the SIE with generalized drift and non-smooth dispersion function  under consideration. 

 Section~\ref{sec01} provides several examples, primarily in the context of three specific setups:   Subsection~\ref{SS:5.1} discusses the case when the drift does not depend on the solution process of the SIE itself; Subsection~\ref{SS:Ex2} treats the situation when the drift is absolutely continuous with respect to Lebesgue measure; and Subsection~\ref{sec3.1}   treats the case of time-homogeneous co\"efficients when the input process is a Brownian motion.   Finally, Subsection~\ref{SS:Ex4} provides an example related to skew Brownian motion.

Section~\ref{comp} presents a comparison result in the spirit of Section~VI.1 in \citet{Ikeda_Watanabe} but using entirely different methods and with quite  broader scope. 
Finally, Section~\ref{cont} establishes under appropriate conditions the continuity of the input-output map in the sense of \citet{Wong_Zakai_a, Wong_Zakai_b} for equations of the type studied in this paper.  Appendix~\ref{S: regul} summarizes aspects concerning the regularization of OIEs by means of additive noise.

\section{Setup, notation, and examples}  
\label{SS:setup}

\subsection{Path space}

We place ourselves on the canonical path space $\,\Omega = C([0,\infty); \R)$    of   continuous functions $\, \omega: [0,\infty)\rightarrow  \R\,$, equipped with the topology of uniform convergence on compact sets,    denote by   $W (\cdot)= \{W(t)\}_{0 \leq t < \infty}$   the co\"ordinate mapping process $\, W(t,\omega) = \omega (t),\, ~0 \le t < \infty\,$ for all $\omega \in \Omega \,$, and consider the filtration  $\,\mathbb{F}^W =  \{ \mathcal{F}^W (t) \}_{0 \le t < \infty}$ with $\,\mathcal{F}^W (t) = \sigma ( W(s), \, 0 \le s \le t)\,$   generated by  $W (\cdot)$;  
this filtration   is left-continuous, but   not right-continuous. We introduce its right-continuous version $\,\mathbb{F}  =  \{ \mathcal{F}  (t) \}_{0 \le t < \infty}$ by setting $\, {\cal F} (t) := \bigcap_{s>t} \mathcal{F}^W (s)\,$ for all $\, t \in [0, \infty)\,$, and define  $\,\mathcal{F} \equiv \mathcal{F}^W(\infty) := \bigvee_{0 \leq t < \infty} \mathcal{F}^W(t)$.

  We shall consider, on the measurable space $(\Omega, \mathcal{F})$, the collection $\, \mathfrak{P}\,$ of {\it semimartingale measures}, that is, of  probability measures $\Prob$  under which the canonical process $W (\cdot)\,$ is a semimartingale in its own filtration  $\mathbb{F}^W$. 
The Wiener measure $\,\Prob_*\,$ is the  prototypical  element of $\, \mathfrak{P}\,$; and under every $\, \mathbb{P} \in  \mathfrak{P}\,$, the canonical process can be thought of as the ``driving noise" of the system that we shall consider.

We fix  an open interval $\, \I = (\ell, r)\,$ of the real line, along with some starting point $\,x_0 \in \I\,$; the interval $\, \I  \,$ will be the state-space of the  solutions to the     equations which are the subject of this work.   We shall denote by $\,\overline{\I}\,$ the one-point compactification of $\,\I\,$, that is, $\,\overline \I = \I \cup \{\Delta\}$ for some $\Delta \notin \I$.  We shall consider also the space $\,\Xi = C_a \big([0,\infty); \overline{\I}\,\big)$  of    $\,\overline{\I}$--valued continuous functions  that get absorbed when they hit the ``cemetery point" $\Delta$. We 
use for all $n \in \N$, $ \, \mathrm{x} \in  C_a([0,\infty);  \overline \I)\,$ the stop-rules \begin{equation}
\label{eq:Sn}
\mathcal{S}_n (\mathrm{x}) \,:=\, \inf \left\{ t \ge 0\,:\,  \mathrm{x} (t) \notin \big(\ell^{(n)}, r^{(n)} \big) \right\}.
\end{equation}
 Here $\{ r_n\}_{n \in \mathbb{N}}$  (respectively, $\{ \ell_n\}_{n \in \mathbb{N}}$)  are some strictly increasing (respectively, decreasing) sequences    satisfying 
$\,
\ell <\ell_{n+1} <  \ell_{n} < x_0 < r_{n}  < r_{n+1}<r
\,$  
for all $\,n \in \mathbb{N}\,$, as well as  $\,\lim_{n \uparrow \infty} \uparrow\,  r_{n}= r$ and $\,\lim_{n \uparrow \infty} \downarrow   \ell_{n} =\ell$.  For later use  we also introduce, for every   path $\,\mathrm{x} \in \Xi\,,$ the stop-rule
\begin{equation}
\label{eq:S}
\mathcal{S} (\mathrm{x}) \,:=  \lim_{n \uparrow \infty} \uparrow \, \mathcal{S}_n (\mathrm{x})\,,  
\end{equation} 
as well as   the following quantities:
	\begin{itemize}
		\item   the  double sequence $\,\big\{ \tau^{(i,n)}\big\}_{(i,n) \in \N_0^2}\,$ of  stop-rules   defined inductively by $\tau^{(0,n)}(\mathrm{x}) = 0$ and $$\tau^{(i+1,n)}(\mathrm{x}) \,=\, \inf\left\{t \geq \tau^{(i,n)}(\mathrm{x}) : \big|\mathrm{x}(t) - \mathrm{x}\big(\tau^{(i,n)}(\mathrm{x}) \big)\big| \geq 2^{-n}\right\} $$ 
for all  $\,(i ,n) \in \N_0^2\,$; the quadratic variation  
		$\langle \mathrm{x} \rangle(\cdot)$ of the path $\mathrm{x}(\cdot)$, defined  
		as
\begin{align}
 \label{QV}
	\langle \mathrm{x} \rangle(t) \,=\,  \liminf_{n \uparrow \infty}\, \sum_{i \in \N}  \left(\mathrm{x}\big(t \wedge \tau^{(i+1,n)}\big) - \mathrm{x}\big(t \wedge \tau^{(i,n)}\big)\right)^2\,, \qquad 0 \le t < \infty,
	\end{align}
 formally  with the convention $\Delta - \Delta = 0\,$;  and 
 		\item   the {\it right} local time   
		$L^\mathrm{x}(\cdot \,, \xi )$  of the path $\mathrm{x}(\cdot)$ at the site $\xi \in \I\,$, defined as
			\begin{align} 
			\label{eq: locTime}
				L^\mathrm{x}(T, \xi ) &\,=\, \limsup_{\varepsilon \downarrow 0} \,\frac{1}{\, 2\,\varepsilon \,} \int_0^T \1_{[\xi, \xi+\varepsilon)}(\mathrm{x}(t)) \, \dx \langle\mathrm{x}\rangle(t)\,, \qquad 0 \le T < \infty\,. 
			\end{align}
	\end{itemize}
We denote by   $\,X (\cdot)= \{ X (t) \}_{0 \leq t < \infty}$ the co\"ordinate mapping process $\, X (t,\mathrm{x}) = \mathrm{x} (t),\, ~0 \le t < \infty\,$ for all $\,\mathrm{x} \in \Xi,$  and introduce the filtration  $\,\mathbb{F}^{ X } =  \{ \mathcal{F}^{ X}  (t) \}_{0 \le t < \infty}$ with $\,\mathcal{F}^{ X}  (t) = \sigma ( X (s), \, 0 \le s \le t)\,$   generated by this new canonical  process $X  (\cdot)\,.$ 

Let us recall from Definition~3.5.15 in \citet{KS1} the notion of  progressive measurability, and adapt it to the present circumstances. We say that a mapping $ \,\mathfrak{M}: [0,\infty) \times C([0,\infty); \R) \times C_a \big([0,\infty); \overline{\I}\,\big) \rightarrow  [-\infty,\infty]\,$ is {\it progressively measurable,} if for every $\, t \in [0, \infty)\,$ its restriction to $ [0,t] \times C([0,\infty); \R) \times C_a \big([0,\infty); \overline{\I}\,\big)\,$ is $\, \mathcal{B }([0,t]) \otimes \mathcal{F}^W  (t) \otimes \mathcal{F}^X  (t) \,/\, \mathcal{B}([-\infty, \infty])$--measurable. 

We observe that   $(t,\mathrm{x} ) \mapsto \langle \mathrm{x} \rangle( t)$ as in \eqref{QV} is a progressively measurable functional on $\,[0,\infty)   \times C_a \big([0,\infty); \overline{\I}\,\big)$,   and that so is the functional		$\, (t,\mathrm{x} ) \mapsto L^\mathrm{x}(  t\,, \xi)$  in \eqref{eq: locTime}  for each given  $\xi \in \I$.  Finally, let us  recall from Subsection~7.14 in \citet{Bichteler_1981},     Subsection~II.a in \citet{Bertoin1987}, and Corollary~VI.1.9 in \citet{RY} that, for any continuous semimartingale $X(\cdot)$,  its quadratic variation can be cast as $ \langle X \rangle (\cdot) = X^2 (\cdot) - X^2 (0) - 2 \int_0^{\, \cdot} X(t)\, \dx X (t)\,,$ and its {\it right} local time has the representation   in \eqref{eq: locTime}.

\subsection{Ingredients of the stochastic integral equation}
  \label{sec22}

In order  to  describe the stochastic integral equation under consideration,  we place ourselves  on the filtered measurable space $(\Omega, \mathcal{F})$, $\,\mathbb{F}  =  \{ \mathcal{F}  (t) \}_{0 \le t < \infty}\,$. We shall fix throughout a measurable function $\,\mathfrak{s}   : \I \ra (0, \infty)\,$  with the property
\begin{equation}
\label{FV}
\log \mathfrak{s} (\cdot) \,~\text{is left-continuous and of   finite first variation on compact subsets of }  \,\I\,. 
\end{equation}
We define
	\begin{align}
\label{eq:J}
	\widetilde{\ell}_{ n } := - \int_{\ell_n}^{x_0} {\dx z \over \,  \mathfrak{s} (z)\,}, \quad \widetilde r_{ n } :=  \int_{x_0}^{r_n} {\dx z \over \,  \mathfrak{s} (z)}\,, \quad \widetilde{\ell} = \lim_{n \uparrow \infty} \widetilde{\ell}_{  n }\,,  \quad \widetilde{r} = \lim_{n \uparrow \infty} \widetilde{r}_{ n }\,,   \quad {\cal J}:=\big( \widetilde{\ell}, \widetilde{r}\,\big) 
\end{align}
and, for all $n \in \N$ and $\mathrm{y} \in  C_a([0,\infty);  \overline\J)$, the stop-rules
\begin{equation}
\label{eq:S2}
 \widetilde{\mathcal{S}}_n (\mathrm{y}) \,:=\, \inf \left\{ t \ge 0\,:\,  \mathrm{y} (t) \notin \big(\widetilde{\ell}_{ n }\,, \widetilde{r}_{ n } \big) \right\},  \qquad \widetilde{\mathcal{S}} (\mathrm{y}) \,:= \lim_{n \uparrow \infty} \uparrow \, \widetilde{\mathcal{S}}_n (\mathrm{y}),
\end{equation}
 where $C_a([0,\infty);  \overline\J)$ is defined in the same manner  as $C_a([0,\infty);  \overline\I)$. 

\smallskip
We shall also fix   a progressively measurable mapping $\, 
\mathfrak{B}: [0,\infty) \times C([0,\infty); \R) \times C_a \big([0,\infty); \overline{\I}\,\big) \rightarrow  \R
\,$  
with $ \mathfrak{B} (0, \cdot, \cdot)=0$.  For instance,     $\mathfrak{B}(\cdot, \cdot, \cdot)$ can be of the form $\mathfrak{B}(T, \cdot, \mathrm{x}) = \int_0^T\mathfrak{b}(\mathrm{x}(t)) \dx t\, $ for some bounded, measurable function $\,\mathfrak{b}: \overline\I \rightarrow \R$, for all $\,T \in [  0, \infty) $ and $\,\mathrm{x} \in   C_a([0,\infty);  \overline{\I})
$.

For any given semimartingale   measure $\, \Prob \in \mathfrak{P}\,$, we shall be interested in the pathwise solvability of SIEs  of the form
\begin{align}
\label{1}
X(\cdot) \,=\, x_0 + \int_0^{\,\cdot} \mathfrak{s} (  X(t)) \big[ \,\dx W(t) + \dx \mathfrak{B} (t,W,X)\,\big] - \int_{\I} L^X (\cdot\, , \xi) \, \mathfrak{s}  (\xi)\, \dx   \frac{1}{\mathfrak{s}  (\xi)}   
\end{align}
on the filtered probability space $\, ( \Omega, {\cal F}, \Prob),\,  \mathbb{F}
=  \{ \mathcal{F}
(t) \}_{0 \le t < \infty}\,$, where the local time $L^X(\cdot\,, \cdot)$ is defined as in \eqref{eq: locTime}. From a systems-theoretic  point of view,   the   process $\, X(\cdot)\,$ represents the ``state'' or ``output'',    and the canonical process $\,W(\cdot)\,$   the     ``input'', of the system with the dynamics of  \eqref{1}. The solution of this equation is defined in general only up until the {\it explosion time} $\, \mathcal{S} (X) \in (0, \infty]$.

More precisely, we have the following formal notions of solvability (Definitions \ref{Def1} and \ref{Def2} below).

\begin{defn}
\label{Def1}
For any given semimartingale   measure $\, \Prob \in \mathfrak{P}\,$, we shall call a process $X (\cdot)$ with values in $C_a \big([0,\infty); \overline{\I}\,\big)$  a {\it solution to the SIE~\eqref{1}} on the filtered probability space $\, ( \Omega, {\cal F}, \Prob),\,  \mathbb{F}
=  \{ \mathcal{F}
(t) \}_{0 \le t < \infty}\,$ up until a stopping time $\, \mathcal{T}\,$ with $\, 0 < \mathcal{T} \le \mathcal{S}(X)\,$, if the following conditions hold  on the stochastic interval $\, [0, \mathcal{T})$: 
\begin{enumerate}[(i)]
	\item the process  $X(\cdot)$ is a continuous $\, \mathbb{F}
	$--semimartingale;	
	\item the process $\mathfrak{B}(\cdot\,, W, X)$ is continuous and of finite first variation on compact subintervals; 
		\item the equation in \eqref{1} holds. \qed
\end{enumerate}
\end{defn}

Point~(iii) in the above definition requires the notion of   stochastic integral. We refer to Section~I.4.d in \citet{JacodS}, and to Section~4.3 in \citet{SV_multi}, for the development of  stochastic integration with respect to a right-continuous  filtration which is not necessarily  augmented by the null sets of the underlying probability measure.  

\begin{defn}
\label{Def2}
By \emph{pathwise solvability} of the SIE~\eqref{1} over a stochastic interval $  [0, \mathcal{T})$ for some stopping time $   0<\mathcal{T} \le {\cal S} (X) $, we mean the existence of a progressively measurable functional $\,\mathfrak{X}: [0,\infty) \times C([0,\infty); \R) \rightarrow \overline \I\,$ such that 
\begin{enumerate}[(i)]
\item  the process $\,X (\cdot) = \mathfrak{X}(\cdot \,,W)$ solves on the interval $  [0, \mathcal{T})$ the SIE~\eqref{1} under {\it any} semimartingale measure $\, \Prob \in \mathfrak{P}\,$; and
\item  the ``input-output" mapping $\,(t, \omega) \mapsto \mathfrak{X}(t, \omega)\,$ is determined   by solving, for each $\, \omega \in C ([0,\infty); \R) $, an appropriate Ordinary (or more generally, Functional) Integral Equation (OIE, or OFE).  \qed
\end{enumerate}
\end{defn}

A solution  as mandated by Definition \ref{Def2}   is obviously {\it strong,} in the sense that  the random variable $X(t)=   \mathfrak{X}( t \,,W)$ is measurable with respect to the sigma algebra $\mathcal F^W(t)\,$  for each $\, t \in [0, \infty) $; and   no stochastic integration with respect to $W(\cdot)$ is necessary for   computing   the input-output mapping $\,\mathfrak{X}(\cdot\,, \cdot)$.

 \begin{rem}
 \label{1/s}
The function  $\mathfrak{s}  (\cdot)$, under the requirements of \eqref{FV},  is  bounded away from both zero and infinity over compact subsets of $\,\I\,$; this is because the condition  \eqref{FV} implies  that the functions   $\,\mathfrak{s}(\cdot) = \exp \big(\log(\mathfrak{s}(\cdot))\big)\,$  and   $\,1/\mathfrak{s}(\cdot) = \exp \big(-\log(\mathfrak{s}(\cdot))\big)\,$  are left-continuous and of finite first variation  on compact subsets   of $\,\I\,$. It follows then from these considerations that 
\begin{align}
\label{LI_new}
{ 1 \over \,\mathfrak{s}  (\cdot)\,} ~\, \text{ is   bounded on compact subsets of }\, \, \I\,.
\end{align}
If the function $\,\mathfrak{s}(\cdot)$ is bounded away from zero and of finite first variation on compact subsets of $\,\I\,$, then $\log \mathfrak{s} (\cdot) \,$ is of finite first variation on compact subsets of $\,\I$. However, if $\mathfrak{s}(\cdot)$ is not bounded away from zero, this implication does not hold; for instance, with $\,\I = \R\,$, and $\mathfrak{s}(x) = 1$ for all  $x \leq 0$, and $\mathfrak{s}(x) = x$ for all $x > 0$, the function  $\mathfrak{s} (\cdot)$   is of finite first variation on compact subsets of $\,\I\,$, but  $\,\log \mathfrak{s} (\cdot)\, $  is not.  Moreover, in the setting under consideration,    the process $\,\mathfrak{s}(X(\cdot))$ is integrable with respect to both $\,\mathfrak{B}(\cdot\,,W, X)$ and   the driving semimartingale $W(\cdot)$.      
 \qed
\end{rem}

 \begin{rem}
 \label{insidious}
Using the   property 
$ \,
\int_{x_0}^{\, \cdot} \mathfrak{s}  (\xi   )\, \dx \big( 1/ \mathfrak{s}  (\xi  )\big)  \,+ \int_{x_0}^{\, \cdot} \big( 1 / \mathfrak{s}  (\xi  +)\big)\, \dx \,\mathfrak{s}  (\xi  ) \equiv  0 \,
$ 
of Lebesgue-Stieltjes integration (e.g., Proposition~0.4.5 in \citet{RY}), we see that the last term in \eqref{1} can be written equivalently as 
\begin{equation}
\label{insid}
  - \int_{\I} L^X (\cdot\, , \xi) \, \mathfrak{s}  (\xi)\, \dx   \frac{1}{\mathfrak{s}  (\xi)} \,=\, \int_\I L^X(\cdot\,, \xi)\,  \frac{\,\dx \mathfrak{s} (\xi)\, }{\mathfrak{s} (\xi +)}\,.
\end{equation}

Consider now the SIE~\eqref{1}  with $ \mathfrak{B} (\cdot\,, \cdot\,, \cdot)\equiv 0\,$,  under a probability measure $\, \mathbb{P} \in  \mathfrak{P}\,$   which renders the canonical process $\, W(\cdot)\,$   a local martingale.  By virtue of \eqref{insid},  we have     the expression 
$$
L^X(T, \xi ) - L^X(T, \xi -) \,=\, L^X(T, \xi ) \, \frac{ \, \mathfrak{s}  (\xi +  ) - \mathfrak{s}  (\xi  ) \, }{\mathfrak{s}  (\xi + )}\,, \qquad \mathbb{P}\text{--a.e.~on the event} \, \,\{ {\cal S} (X) >T\}
$$
for the jump of the local time of   $\,X(\cdot)\,$ at the site $\, \xi \in \I$ and at time $T \geq 0$; here, we are using    $$\,L^X(T, \xi ) - L^X(T, \xi -) = \int_0^T \1_{ \{ X(t) = \xi \} } \, \dx V(t) \,,$$ a basic    property of local time for a continuous semimartingale $X(\cdot)=X(0) + M (\cdot)+ V (\cdot)\,$ (for instance, 
Theorem~VI.1.7 in \citet{RY}). This leads to the ``balance equation"  $$ \mathfrak{s}  (\xi + )\, L^X(T, \xi -) =  \mathfrak{s}  (\xi   )\, L^X(T, \xi  ),$$  and   expresses the {\it symmetric local time} $\,  \widehat{L}^X(T, \xi ) := \big( L^X(T, \xi ) + L^X(T, \xi -) \big) \big/2\,$ as 
 \begin{align*}
 \widehat{L}^X(T, \xi ) \,=\, \frac{1}{\,2\,} \left( 1 + \frac{\,  \mathfrak{s} (\xi)\, }{\,\mathfrak{s} (\xi +)\,} \right)   L^X(T, \xi )\,,  \qquad \mathbb{P}\text{--a.e.~on the event} \, \,\{ {\cal S} (X) >T\} 
 \end{align*}
for each $T \geq 0$. \qed
\end{rem}

\subsection{Stratonovich interpretation}
 \label{sec03}

The SIE~\eqref{1} can be cast in the Stratonovich form
\begin{equation}
\label{straton}
X(\cdot) \,=\, x_0 + \int_0^{\,\cdot} \mathfrak{s} (X(t))  \circ  \dx W(t) + \int_0^{\,\cdot} \mathfrak{s} (  X(t))\, \dx  \mathfrak{B} (t, W, X)       
\end{equation}
when the dispersion fuction $\,\mathfrak{s} (\cdot)\,$ is the {\it difference of two convex functions}, i.e.,  can be written as the primitive of some real-valued function $\, \mathfrak{r} (\cdot)\,$ with finite first variation on compact subsets of $\, \I\,$: namely, as $$\mathfrak{s} (x)=  \mathfrak{s} (c) + \int_c^x \mathfrak{r} (\xi)\, \dx \xi\,, \qquad x \in \I$$ for some $c \in \I$. For   convenience we shall  adopt  the convention that the function $\, \mathfrak{r} (\cdot)\,$  is  left-continuous.

Indeed, in this case the process $\,  \mathfrak{s} (X(\cdot))\,$ is a continuous semimartingale with decomposition
\begin{align*}
 \mathfrak{s} (X(\cdot)) &-  \mathfrak{s} (x_0)  = \int_0^{\, \cdot}  \mathfrak{r} (X( t))\,\dx X(t) + \int_\I L^X(\cdot\,, \xi)\, \dx  \mathfrak{r} (\xi) \\
&=   \int_0^{\, \cdot}  \mathfrak{r} (X( t))\, \mathfrak{s} (X( t))\,\big[ \,  \dx W(t) +   \dx  \mathfrak{B} (t, W, X)\, \big] + \int_\I L^X(\cdot\,, \xi) \left[ \dx  \mathfrak{r} (\xi) - \mathfrak{r} (\xi) \,\mathfrak{s} (\xi)\,  \dx \frac{ 1}{\mathfrak{s} (\xi) }\right] 
\end{align*}
by the generalized It\^o--Tanaka formula. 
Therefore, the Stratonovich and It\^o integrals are then related via 
\begin{align*}
\int_0^{\,\cdot} \mathfrak{s} (X(t))  \circ  \dx W(t) & - \int_0^{\,\cdot} \mathfrak{s} (X(t))     \dx W(t)  = { 1 \over \,2\,} \int_0^{\,\cdot} \mathfrak{r} (X(t))  \, \mathfrak{s} (X(t)) \, \dx \langle W \rangle ( t)\,=\, { 1 \over \,2\,} \int_0^{\,\cdot} \frac{\mathfrak{r} (X(t))}{\mathfrak{s} (X(t))} \, \dx \langle X \rangle ( t)\\
&= \int_\I L^X(\cdot\,, \xi)\,  \frac{\,\mathfrak{r} (\xi)}{\mathfrak{s} (\xi)}   \dx \xi \,=\, \int_\I L^X(\cdot\,, \xi)\,  \frac{\,\dx \mathfrak{s} (\xi)\, }{\mathfrak{s} (\xi)}\,=\, - \int_{\I} L^X (\cdot\, , \xi) \, \mathfrak{s}  (\xi)\, \dx \frac{ 1}{\mathfrak{s}  (\xi)  } \,,
\end{align*}
and \eqref{straton} follows from \eqref{1}. We have used here the occupation-time-density property of semimartingale local time; see, for instance, pages~224-225 in \citet{KS1}, as well as Definition~3.3.13 there. 
These considerations allow the interpretation of the last integral in \eqref{1} as a ``singular Stratonovich-type correction term."

\section{Pathwise solvability}
 \label{sec02}

The possibility that an SIE such as that of \eqref{1}
might be solvable pathwise, is suggested by the following observation: If we work under the Wiener measure $\, \Prob_*\,$,  if the function $\mathfrak{s} (\cdot)\,$ is   continuous and continuously differentiable, and if $\,\mathfrak{B}(T, \cdot, \mathrm{x}) = \int_0^T \mathfrak{b}(t,\mathrm{x}(t)) \dx t\,$ for some bounded measurable function $\mathfrak{b}: [0,\infty) \times \overline \I \rightarrow \R$ for all $T \geq 0$ and $\mathrm{x} \in C_a([0,\infty), \overline \I)$,  then,  on the strength of the occupation-time-density property of semimartingale local time,  the corresponding SIE
\begin{align}
\label{1c}
X(\cdot) \,=\, x_0 + \int_0^{\,\cdot} \mathfrak{s} (  X(t))\, \big[ \,\dx W(t) + \mathfrak{b} (t, X(t)) \,\dx t \,\big] - \int_{\I} L^X (\cdot\, , \xi) \, \mathfrak{s}  (\xi)\, \dx \frac{1}{ \mathfrak{s}  (\xi) } 
\end{align}
of \eqref{1}  takes  the familiar  form
\begin{equation}
\label{1a}
X(\cdot) \,=\, x_0 + \int_0^{\,\cdot} \mathfrak{s} (X(t))  \left[ \,\dx W(t) + \Big( \mathfrak{b} (t, X(t)) + { 1 \over \,2\,} \, \mathfrak{s}^\prime (X(t))\Big) \,\dx t\,\right]      ;
\end{equation}
\noindent
whereas, if  the function $\,\mathfrak{s} (\cdot)\,$ is   twice continuously differentiable, this equation can be cast  in terms of {Stratonovich} stochastic integration as 
\begin{equation*}
X(\cdot) \,=\, x_0 + \int_0^{\,\cdot} \mathfrak{s} (X(t))  \circ  \dx W(t) + \int_0^{\,\cdot} \mathfrak{s} (  X(t))\,\mathfrak{b} (t, X(t))    \,\dx t  \, .
\end{equation*}
 From the results of \citet{Doss1977} and \citet{Sussmann1978} we know  that, at least in the case $ \,\I = \R\,$, solving this latter SIE~\eqref{1a} amounts  to solving pathwise an ordinary   integral equation in which the  source of randomness, that is, the $\, \Prob_*$--Brownian motion $W(\cdot),$  appears only parametrically  through its co\"efficients,   {\it not in terms of stochastic integration;} see the OIE~\eqref{7}   below. This ``classical" theory is also covered in books, for instance in Section~III.2 of \citet{Ikeda_Watanabe},     Section~5.2.D in \citet{KS1}, or Chapter~2 in \citet{LyonsQian2002}.

\subsection{Basic properties of the space transformation}

For each $c \in \I$,  we   define the strictly increasing function $\, H_c : (\ell,r) \rightarrow (-\infty, \infty)\,$   by 
\begin{equation}
\label{3}
H_c(x) \,:=\, \int_c^x {\dx z \over \,  \mathfrak{s} (z)\,}, \qquad x \in \I.
\end{equation}
Here and in what follows, for the context of Lebesgue-Stieltjes integration, we define $\,\int_y^y f(z) \dx z = 0\,$ as well as  $\,\int_y^x f(z) \dx z = - \int_x^y f(z) \dx z$ for all $\,(x,y) \in \R^2\,$ with $\,x < y\,$ and an appropriate function $f$.
We note that the  function $H_c(\cdot)$ is indeed well-defined for all $\,c \in \I\,$, thanks to \eqref{LI_new}.
Next, for each $\, c \in \I\,$, we set
\begin{align} \label{eq:lc,rc}
	\widetilde{\ell} (c) := H_c ( \ell + ) \,:=\, \lim_{x \downarrow \ell} H_c (x) \in [-\infty, \infty), \qquad 
\widetilde{r}(c) := H_c ( r - ) \,:=\, \lim_{x \uparrow r} H_c (x) \in (-\infty, \infty]. 
\end{align}
We   note   $\, {\cal J}  =(\widetilde{\ell} (x_0)\,,\, \widetilde{r} (x_0)\big)\,$ in the notation of \eqref{eq:J}, and consider   the domain
\begin{align}  
\label{eq:D}
	\mathcal{D}\, := \,\left\{(c,w) \in \I \times \mathbb{R} \,:\,  \,\,  w \in \big(\widetilde{\ell} (c)\,,\, \widetilde{r} (c)\big) \right\} .
\end{align}

For later use,   we  observe that
\begin{equation}
\label{4}
D H_c (x) \,=\, { 1 \over \,\mathfrak{s} (x)\,}\, , \qquad (c,x)   \in \I^{\,2},
\end{equation}
where  the symbol $\, D\,$ stands for differentiation with respect to the second, parenthetical argument. The derivative is considered in its left-continuous version.
The inverse function $\, \Theta_c : \big( \widetilde{\ell} (c), \widetilde{r} (c) \big) \rightarrow \I \, $ of $H_c(\cdot)$
is also well-defined for each $c \in \I$, thanks to the strict positivity of the function $\mathfrak{s}(\cdot)$, and satisfies $ \Theta_c (0) \,=\, c$ as well as 
\begin{equation}
\label{5}
  D \Theta_c  (w) \,=\, \mathfrak{s} \big(\Theta_c (w) \big),  \qquad (c, w)  \in  {\cal D}.
\end{equation}
Once again, the derivative is considered in its left-continuous version.  
In particular, for each given $\, c \in \I$, the functions $\, x \mapsto H_c (x)\,$  and  $\, w \mapsto \Theta_c (w)\,$ are strictly increasing.
\begin{rem} 
\label{R:ODE}
We note that, for each $c \in \I$, the function $\Theta_c(\cdot)$ solves the OIE 
\begin{align} 
\label{eq: OIE Theta}
\Theta_c (w) = c + \int_0^w \, \mathfrak{s} \big( \Theta_c (\zeta) \big)  \dx \zeta\,, \qquad w \in \big(\widetilde{\ell} (c), \widetilde{r} (c)\big)\,.  
\end{align}
The function $\, \Theta_c (\cdot)\,$ {\it is actually the only solution of the integral equation}  \eqref{eq: OIE Theta}; this is because 
  {\it any} solution $\vartheta(\cdot)\,$ of  the equation~\eqref{eq: OIE Theta}  satisfies 
    $$
H_{c} ( \vartheta(w))= \int_{\vartheta(0)}^{\vartheta(w)   } \frac{1}{\mathfrak{s} (\zeta)}\, \dx \zeta  = \int_{0}^{w} \frac{1}{\mathfrak{s} (\vartheta(u))} \,\dx \vartheta(u) = w\,, \qquad w \in \big(\widetilde{\ell} (c), \widetilde{r} (c)\big), 
$$   
therefore  $\,\vartheta(\cdot)\equiv\Theta_c(\cdot)$. \qed
\end{rem}

 Let us also observe that the additivity property $\, H_c (\lambda) + H_\lambda (\xi) = H_c (\xi)\,$ for all $(c, \lambda,\xi) \in \mathcal{I}^{\,3}$, fairly evident from \eqref{3},  translates into the {\it composition property} 
\begin{equation}
\label{compo}
\Theta_{\Theta_{c} (\gamma)} (w) \,=\, \Theta_c \big( \gamma + w \big)\,, \qquad c\in \mathcal{I}, ~~ \gamma  \in \big( \widetilde{\ell} (c), \widetilde{r} (c) \big), ~~\gamma+ w  \in \big( \widetilde{\ell} (c), \widetilde{r} (c) \big).
\end{equation} 

 Finally  we note that, thanks to the observations in Remark~\ref{1/s},   both functions $\, H_c (\cdot)\,$ and $\, \Theta_c (\cdot)\,$ can be expressed as   differences of two convex functions for each $\, c \in \mathcal{I}\,$.

\subsection{Preview of results}  
\label{SS:preview}

The   {\it Lamperti-type approach}  reduces the problem of solving the SIE~\eqref{1}  under an arbitrary semimartingale measure $\, \Prob \in \mathfrak{P}\,$  to that of solving,  for  all ``relevant'' paths $\,\omega  \in C([0,\infty); \R)\,$,  an Ordinary  Functional Equation (OFE) of the form
\begin{align}
\label{7L}
\Gamma ( t) \,=\,  \mathfrak{B} \big( t\,,   \omega , \Theta_{x_0} \big(  \Gamma + \omega   \big) \big) \,, \qquad 0 \le t < \widetilde{\mathcal S}(\Gamma+\omega)
\end{align}
 in the notation of \eqref{eq:S2},   and then produces a solution of  the SIE \eqref{1} in the notation of \eqref{eq:S},   simply through the pointwise evaluation
\begin{equation}
\label{8L}
X( t) \,:=\,  \Theta_{x_0} \big( \Gamma (t) + W( t)   \big)\,, \qquad 0 \le t < \mathcal{S} (X) = \widetilde{\mathcal{S}}(\Gamma+W)\,.
\end{equation}

The {\it Doss-Sussmann-type approach} relies on the following observation.
  Given a function $\Gamma(\cdot)$ that satisfies the OFE \eqref{7L},  we    can define   the function $\,	C( t) : = \Theta_{x_0} \big(\Gamma( t)\big)\,$,  $\, 0 \le t < \mathcal{S}(C) = \widetilde{\mathcal S} (\Gamma)\,$ 
  and note that    it satisfies an OIE of the form
\begin{align}
\label{7}
C(\cdot) \,=\, x_0 + \int_0^{\, \cdot} \mathfrak{s} \big( C(t) \big) \, \dx \mathfrak{B} \left(t,   \omega, \Theta_{C(\cdot)} ( \omega(\cdot)   ) \right) ,
\end{align}
thanks to the composition property in \eqref{compo}, 
at least up until the first time
\begin{equation}
\label{R}
{\cal R} (C,\omega) \,:=\, \lim_{n \uparrow \infty} \uparrow {\cal R}_n (C,\omega)
\end{equation}
the two-dimensional path $\, (C(\cdot), \omega(\cdot))\,$ exits the domain of \eqref{eq:D}; here we have denoted 
\begin{align}
{\cal R}_n (C,\omega)\, :=& \inf \big\{ t \ge 0 : C(t) \notin (\ell_{n}, r_{n}) ~~\text{or}~~ \omega (t) \notin \big( H_{C(t)} (\ell_{n}  ), H_{C(t)} (r_{n} )\big) \big\} \nonumber\\
=& \inf \big\{ t \ge 0 : C(t) \notin (\ell_{n}, r_{n}) \big\} \wedge \, \inf \big\{ t \ge 0 : \Theta_{C(t)} (\omega (t) ) \notin (\ell_{ n }, r_{ n }) \big\}  \nonumber \\
=&\,  {\cal S}_n (C) \wedge {\cal S}_n \big(\Theta_{C(\cdot)} (\omega (\cdot) ) \big) \label{Rn}
\end{align}
in the manner of \eqref{eq:Sn}.  Conversely, any $\mathbb F$--adapted  solution $C(\cdot)$ to the the OIE in \eqref{7} produces a solution of  the SIE~\eqref{1}  simply through the pointwise evaluation
\begin{equation}
\label{8}
X( t) \,:=\,  \Theta_{C( t)} ( W( t)   )\,, \qquad 0 \le t < {\cal R} (C,W) = \mathcal{S} (X)    \wedge \mathcal{S} (C)\,;
\end{equation}
the last equality here is obvious from \eqref{R}, \eqref{Rn} and from the definition of the process $\,X(\cdot)\,$ in \eqref{8}.

\subsection{Relating the SIE to a family of OIEs}
We are now ready to state and prove the first main result of this work.

\begin{thm}[{{\bf A Lamperti-type result}}]
\label{T1L} 
 For any given semimartingale measure $\Prob \in \mathfrak{P}\,$, the following  
hold:
\begin{enumerate}
\item[(i)]
Given any   solution $\,X(\cdot)\, $   of the stochastic integral equation \eqref{1} on the filtered probability space $\, ( \Omega, {\cal F}, \Prob),\,  \mathbb{F} =  \{ \mathcal{F}(t)  \}_{0 \le t < \infty}\, $ up until the explosion time $\, \mathcal{S} (X)\,$,   the process 
  \begin{equation}
\label{Y}
 Y ( t)\, := H_{x_0} \big(X (  t ) \big)\,,\qquad 0 \le t <  
 \mathcal{S} (X)  
 \end{equation}
 is well-defined up until its own explosion time $\,\widetilde{\cal S} (Y) =  {\cal S} (X)\, $ as in    \eqref{eq:J}, \eqref{eq:S};   and the process $\, \Gamma (\cdot) := Y(\cdot) - W(\cdot)\,$ is of finite first variation on compact intervals  and solves the OFE 
\begin{equation}
\label{2L}
\Gamma ( t) \,=\,  \mathfrak{B} \big(t,   W   , \Theta_{x_0}  ( \Gamma   + W)  \big) \,, \qquad 0 \le t < \widetilde{\mathcal S}(\Gamma+W).
\end{equation}

   \item[(ii)]
Conversely, suppose we are given an   $\mathbb{F}$--adapted process $\,\Gamma (\cdot)\,$ of finite first variation on compact intervals, defined up until the  stopping     time $\, \widetilde{\mathcal S}(\Gamma+W)\,$ and solving, for $\Prob$--almost every path $\,W(\cdot)\,$,  the OFE of \eqref{2L} up until    $ \, \widetilde{\mathcal S}(\Gamma+W)\,$.

 Then   the process $\,X(\cdot):= \Theta_{x_0} \big( \Gamma ( \cdot) + W( \cdot)   \big)\,$    is $\,\mathbb{F}$--adapted and solves the stochastic integral equation \eqref{1} on the filtered probability space $\, ( \Omega, {\cal F}, \Prob),\,  \mathbb{F}=  \{ \mathcal{F}(t) \}_{0 \le t < \infty}\, $, up until the explosion   time  $\,   \mathcal{S} (X)= \widetilde{\mathcal S}(\Gamma+W)$.  In particular, if the process $\Gamma(\cdot)\,$ is $\,\mathbb{F}^W$--adapted, so is  $X(\cdot)$.

\end{enumerate}
\end{thm}

\begin{proof}
We organize the proof in two steps. The first one is an analysis, showing the statement in part~(i) of the theorem; the second one is a synthesis, proving the statement in part~(ii).

\noindent
{\sl Analysis:}  We  start by assuming  that such a solution $\,X(\cdot)\,$ to the SIE~\eqref{1}, as postulated in part~(i) of the theorem, has been constructed on the filtered probability space $\, ( \Omega, {\cal F}, \Prob),\,  \mathbb{F} =  \{ \mathcal{F} (t) \}_{0 \le t < \infty}\,$ for the given semimartingale measure $\, \Prob \in \mathfrak{P}\,$, up until the explosion time $\, {\cal S} (X)$. Then the process $Y(\cdot) = H_{x_0} \big( X(\cdot) \big)\,$ of \eqref{Y} is well-defined up until the explosion time 
\begin{equation*}
	 {\cal S} (X)  = \widetilde{\cal S} (\Gamma +W)=  \widetilde{\cal S} (Y)\,,
\end{equation*}
as follows   from the definition of the function $H_{x_0}(\cdot)$ in \eqref{3}.  Finally, we recall that $H_{x_0}(x_0)=0$ holds and that the function $1/\mathfrak{s}(\cdot)$  is of finite first variation on compact subsets of $\,\I\,$ (recall the discussion in Remark~\ref{1/s}). Thus   $\, H_{x_0}(\cdot)\,$ is the difference of two convex functions and the {It\^o--Tanaka} rule  (e.g., Theorem 3.7.1 in Karatzas and Shreve (1991)) gives 
$$
Y(T) \,= \,\int_0^{T}   D H_{x_0}  ( X(t) )\, \dx X(t) +     \int_\I   \,L^X (T,  \xi )    \, \dx \, D H_{x_0} (  \xi) \qquad \text{on}~~ \{ \mathcal{S}(X)>T\}
$$
for all $\, T \in [0, \infty)$.  Now it is rather clear from \eqref{1} and  \eqref{4}  that the first of these integral terms is 
\begin{align*}
\int_0^{T} { 1 \over \,  \mathfrak{s} (  X(t) ) \,} \cdot \mathfrak{s} (  X(t))  \left[ \,\dx W(t) + \dx \mathfrak{B} (t, W,X) \right] & - \int_0^{T} { 1 \over \,  \mathfrak{s} (  X(t) ) \,} \int_\I   \, L^X (\dx t , \xi) \, \mathfrak{s}  (\xi) \, \dx  \frac{1}{\mathfrak{s}(\xi) } \\
& =\, W(T) +  \mathfrak{B} ( T, W,X ) - \int_\I L^X (T,  \xi ) \,\dx \frac{1}{\mathfrak{s}(\xi) } \,;  
\end{align*}
 and that the second integral term  is  
$\,
    \int_\I L^X (T,  \xi ) \,\dx \big(  1 / \mathfrak{s}(\xi) \big)\,$.  Combining these two terms we obtain  
 \begin{align*}
 	\Gamma(T) = Y(T) - W(T) = \mathfrak{B} ( T, W,X ) = \mathfrak{B} \big( T, W, \Theta_{x_0}(Y) \big) = \mathfrak{B} \big( T, W, \Theta_{x_0}\big(\Gamma + W\big) \big) ,
 \end{align*}
which yields all the claims in part~(i) of the theorem.

\noindent
{\sl Synthesis:} We place ourselves   on the filtered probability space $\, ( \Omega, {\cal F}, \Prob),\,  \mathbb{F}
=  \{ \mathcal{F}
(t) \}_{0 \le t < \infty}\,$ for the given semimartingale measure $\, \Prob \in \mathfrak{P}\,$. As postulated in part (ii) of the theorem, we assume  that the OFE~\eqref{2L} has an $\mathbb{F}
$--adapted  solution $ \Gamma(\cdot)$  up until the stopping time $\widetilde{\cal S} (\Gamma +W)$. 
With the process $X(\cdot) = \Theta_{x_0} ( \Gamma ( \cdot) + W( \cdot)  )$, it is clear  that 
$	 {\cal S} (X) =\widetilde{\cal S} (\Gamma +W)$,  
and the {It\^o--Tanaka} rule   gives      now  
\begin{equation}
\label{ItoTan}
X(T) \,= \,x_0 + \int_0^{T}  D \Theta_{x_0}  \big(\Gamma(t) +  W(t) \big)\, \dx \big(\Gamma(t) + W(t)\big)  +   \int_\J   L^{\Gamma+W} (  T ,    y )\,  \dx \,D \Theta_{x_0} (  y)
\end{equation}
on $\, \{ \mathcal{S}(X) >T \}\,$, for all $\, T \in (0, \infty)\,$. On the strength of \eqref{5}, the first integral in this expression is
$$
\int_0^{T} \mathfrak{s} \big( \Theta_{x_0} (\Gamma(t) +  W(t)) \big)\,  \dx \big[   \mathfrak{B} ( t, W,X )  +  W(t)\big]\,=\, \int_0^{T} \mathfrak{s} \big( X(t)  \big)\,  \dx \big[   \mathfrak{B} ( t, W,X )  +  W(t)\big]\,.
$$
As for   the second integral in \eqref{ItoTan}, we recall   the property 
$$
L^{\Gamma + W } (T, y) = { 1 \over \,D^+ \Theta_{x_0}  (y)\,}\, L^X \big(  T, \Theta_{x_0} (y) \big)  
$$
from Exercise~1.23 on page~234 in \citet{RY}; we denote here by $\,D^+ \Theta_{x_0} (\cdot)\,$ the {\it right}-derivative of the function  $\,  \Theta_{x_0} (\cdot)\,$, namely  $\,D^+ \Theta_{x_0} (y) = \mathfrak{s} \big( \Theta_{x_0} (y) + \big)   $ from \eqref{5}. These considerations  allow  us to cast the second integral in \eqref{ItoTan} as 
 \begin{align*}
 \int_\J L^{\Gamma+W} (  T, y )\, \dx \, D \Theta_{x_0} (   y) &=   \int_\J L^X \big(  T, \Theta_{x_0} (y) \big)\,   { \,\dx \, D\Theta_{x_0}  (  y)\, \over \, D^+ \Theta_{x_0}  (y)\,}   =  \int_\J L^X \big(  T, \Theta_{x_0} (y) \big)\, {\, \dx \,\mathfrak{s} (\Theta_{x_0} (y)  )\, \over \, \mathfrak{s} (\Theta_{x_0} (y) +)\,} \\
&= \int_\I L^X (  T, \xi )\, {\, \dx \,\mathfrak{s}  (\xi  )\, \over \, \mathfrak{s}  (\xi  +)\,}\,=\, - \int_\I L^X (T , \xi) \, \mathfrak{s}  (\xi)\, \dx \frac{1}{  \mathfrak{s}  (\xi) } \,,
\end{align*}
where the last equality follows from \eqref{insid}.  All in all, we conclude    that the process $\, X(\cdot)\, $ solves the SIE~\eqref{1} on the stochastic interval $[0,\mathcal{S}(X)) $. This completes the proof of part~(ii) of the theorem.
\end{proof}

\begin{rem}[Possible generalizations of Theorem~\ref{T1L}]  We have assumed that the function $\log \mathfrak{s}(\cdot)$ is of finite first variation on compact subsets of $\,\I\,$. The   question arises: How much of the pathwise approach of Theorem~\ref{T1L} goes through,   if we only   assume  that the function $1/\mathfrak{s}(\cdot)$ is  simply integrable  on compact subsets of $\,\I\,$?  
As a first observation, the two functions $H_c(\cdot)$ and $\Theta_c(\cdot)$  will then not be   expressible    necessarily     as   differences of two convex functions; they will only be absolutely continuous with respect to Lebesgue measure.  Therefore, by the arguments in \citet{Cinlar1980}, we cannot expect then the continuous process $\,\Theta_{x_0}(\Gamma(\cdot) + W(\cdot))$ to be a semimartingale. 

We also note that the second integral term in \eqref{1} is not defined if the function $\log \mathfrak{s}(\cdot)$ is not of finite first variation. However, we might formally apply integration-by-parts to that integral and obtain an integral of Bouleau--Yor type, derived in \citet{Bouleau1981}; see, for example, \citet{Ghomrasni}.  The computations in \citet{Wolf_transformations} then will let us expect that $\Theta_{x_0}(\Gamma(\cdot) + W(\cdot))$ is a \emph{local Dirichlet} process with a zero-quadratic variation term of Bouleau--Yor  type. Dirichlet processes were introduced in \citet{F1981Dirichlet} and were studied by   \citet{Bertoin1986, Bertoin1987},  \citet{Fukushima1994}, among many others.  For stochastic differential equations involving Dirichlet processes, we refer to \citet{Engelbert_Wolf},  \citet{FlandoliRussoWolf1, FlandoliRussoWolf2} and \citet{Coviello2007}.

The proof of Theorem~\ref{T1L} relies on the It\^o--Tanaka formula.  Much work  has been done  on obtaining more general change-of-variable formulas that can accommodate Dirichlet processes as inputs and/or outputs. Here, we refer to \citet{Wolf_ito}, \citet{Dupoiron}, \citet{Bardina2007}, \citet{Lowther2010}, and the many references therein. To the best of our knowledge, it is an open problem to connect these techniques and generalize our approach here to the situation when the function $1/\mathfrak{s}(\cdot)$  is  only integrable on compact subsets of $\,\I$. 
A related open problem is to generalize the class of input processes $W(\cdot)$ from the class of all semimartingales   
to the larger class of all local Dirichlet processes. A first step in this direction, for smooth coefficients, was made in \citet{Errami2002}, and for so-called \emph{weak Dirichlet processes}  in    \citet{Errami2003}.
\qed
\end{rem}

\begin{cor}
\label{cor2}
	We fix a semimartingale measure $\,\Prob \in \mathfrak{P}\,$. The SIE~\eqref{1} has at most one solution, if and only if the OFE~\eqref{7L}  has at most one $\mathbb{F}$--adapted solution for $\Prob$--almost every path $W(\cdot)$.   
	Furthermore, the SIE~\eqref{1} has a  solution, if and only if the OFE~\eqref{7L}  has an $\mathbb{F}$--adapted solution for $\Prob$--almost every path $W(\cdot)$.
\end{cor}

\begin{cor} 
\label{cor3}  
	We   assume that  there exists an $\mathbb{F}^W$--adapted process $\, \Gamma (\cdot)\,$, defined up until the stopping  time $\widetilde{\mathcal{S}} (\Gamma +W)\,$ of  \eqref{eq:S}, of finite first variation on compact subintervals of $\,\big[0,\widetilde{\mathcal{S}} (\Gamma +W)\big)\,$  and solving the OFE~\eqref{7L} for each $\omega  \in \Omega$. 
Then Theorem~\ref{T1L} guarantees that the SIE~\eqref{1} is pathwise solvable in the sense of Definition \ref{Def2}, with $\, \mathfrak{X} ( t , \omega) = \Theta_{x_0}  ( \Gamma (  t) + \omega (t)  ) \, $ for all $( t, \omega) \in [0, \infty) \times \Omega$. 
\end{cor}

\begin{rem} [{{\it Pathwise Stochastic Integration}}] 
In the setting of Corollary \ref{cor3},  the SIE~\eqref{1} can   be cast on the strength of   Theorem~\ref{T1L} as 
\begin{align}
\int_0^{T} \mathfrak{s} \big(\Theta_{x_0} ( \Gamma (t)+ W( t))\big)\, & \dx W(t)   =  \,\Theta_{x_0} \big( \Gamma (T)+W( T) \big)  - x_0 -\int_0^{T} \mathfrak{s} \big(   \Theta_{x_0} ( \Gamma (t)+ W( t)   )\big)   \dx \Gamma(t)  \nonumber\\
&+ \int_{\I} L^{\Theta_{x_0} ( \Gamma  +W    )} (T , \xi)\,\mathfrak{s}  (\xi)\, \dx \, \frac{1}{\mathfrak{s}  (\xi) }
\,, \qquad 0 \le T <\widetilde{\mathcal{S}} (\Gamma +W).
\label{eq:1SIL}
\end{align}
We note that the right-hand side of \eqref{eq:1SIL} is  defined path-by-path, and is an $\mathbb{F}^W$--adapted process. Moreover, these equalities hold  under any semimartingale measure $\Prob \in \mathfrak{P}\,$ (at least  $\Prob$--almost surely, as stochastic integrals are defined only thus). Consequently, 
the identification \eqref{eq:1SIL}  corresponds to a pathwise definition of the stochastic integral on its left-hand side. This construction yields a version of the stochastic integral that is not only $\mathbb{F}-$adapted but also $\mathbb{F}^W$--adapted.
We   refer to \citet{F1981}, \citet{Bichteler_1981}, 
\citet{Karandikar1995}, \citet{STZ_aggregation},   \citet{Nutz2012}, and \citet{PerkowskiPromel} for   general results on pathwise  stochastic integration.
\qed
\end{rem}

\begin{cor}[{\bf A Doss-Sussmann-type result}]
\label{T1}
For any given semimartingale measure $\,\Prob \in \mathfrak{P}$, the following   hold:
\begin{enumerate}
\item[(i)]
Given any   solution $\,X(\cdot)\, $   of the stochastic integral equation \eqref{1} on the filtered probability space $\, ( \Omega, {\cal F}, \Prob),\,  \mathbb{F} =  \{ \mathcal{F}(t)  \}_{0 \le t < \infty}\, $ up until the explosion time $\, \mathcal{S} (X)\,$,   the process 
  \begin{equation*}
 C( t)\, := \Theta_{x_0} \big( \mathfrak{B} (  t\,,W,X) \big)\,,\qquad 0 \le t <  
\mathcal{R} (C,W)= 
\mathcal{S}(X)\wedge \mathcal{S} (C)
 \end{equation*}
 is well-defined,   and is the unique     $\mathbb{F}$--adapted solution of the ordinary integral equation   
\begin{equation}
\label{2}
C(  t) \,=\, x_0 + \int_0^{\, t}  \mathfrak{s} ( C(u) ) \, \dx  \mathfrak{B} (u, W,X    )\,,\qquad 0 \le t <  \mathcal{R} (C,W)= 
\mathcal{S}(X)\wedge \mathcal{S} (C).
\end{equation}
 Moreover, with this definition of the process $C(\cdot)$  we have once again 
\eqref{8}, namely 
$$\,
X( t)= \Theta_{C( t)} \big( W(  t)   \big) =    \Theta_{x_0} \big( \mathfrak{B} (   t\,,W,X) + W( t) \big)
 \,,\qquad 0 \le t < \mathcal{R} (C,W)= 
\mathcal{S}(X)\wedge \mathcal{S} (C) ,
$$
as well as    the ordinary integral equation \eqref{7} for $\Prob$--almost each path $\,W(\cdot)\,$.

 \item[(ii)]
Conversely, suppose we are given an  $\mathbb{F}-$adapted process $\,C(\cdot)\,$ defined up until the explosion time $\, \mathcal{S} (C)\,$ as in \eqref{eq:S},   and solving the ordinary integral equation \eqref{7} for $\Prob-$almost every path $\,W(\cdot)\,$ up until the stopping time $\, \mathcal{R} (C,W)\,$ of \eqref{R}.

Then   the process $\,X(\cdot)= \Theta_{C( \cdot)} ( W( \cdot)   )\,$  of \eqref{8} is $\,\mathbb{F}-$adapted and solves the stochastic integral equation \eqref{1} on the filtered probability space $\, ( \Omega, {\cal F}, \Prob),\,  \mathbb{F}=  \{ \mathcal{F}(t)  \}_{0 \le t < \infty}\, $, up until the stopping  time  $\,     \mathcal{S} (C) \wedge 
\mathcal{S} (X)= \mathcal{R} (C,W)\,$.  In particular, if the process $\,C(\cdot)\,$ is $\,\mathbb{F}^W$--adapted, so is  $X(\cdot)$.
   
\end{enumerate}
\end{cor}
\begin{proof}
	The corollary is a direct consequence of Theorem~\ref{T1L}. More precisely, to establish part~(ii), we define $\Gamma (\cdot) := H_{x_0}(C(\cdot))$ and note that $\Gamma(0) = 0$ and
\begin{align*}
	  \dx \Gamma ( t)\, &= \,\dx H_{x_0} \big(C( t)\big) =  \dx  \mathfrak{B}\left( t \, , W (\cdot), \Theta_{C(\cdot)} \big(W(\cdot)\big)\right) =   \dx \mathfrak{B}\left( t \, , W (\cdot), \Theta_{\Theta_{x_0}(\Gamma(\cdot))}\big(W(\cdot)\big)\right)\\
	     &=\,  \dx  \mathfrak{B}\big( t \, , W (\cdot), \Theta_{x_0}\big(\Gamma(\cdot)+W(\cdot)\big)\big), \qquad 0 \leq t < \cal R(C,W)
\end{align*}
Here the second equality follows from \eqref{4} and \eqref{7},  and the last equality   from the composition property in \eqref{compo}. Therefore,  the process $\Gamma(\cdot)$ satisfies the OFE in \eqref{2L} .
 Moreover, we note   
 \begin{align*}
 	X( \cdot) = \Theta_{C(\cdot)} \big(W( \cdot)\big) =  \Theta_{\Theta_{x_0}(\Gamma(\cdot))}\big(W( \cdot)\big) =  \Theta_{x_0} \big( \Gamma( \cdot) + W( \cdot) \big),
 \end{align*}
  by the composition property \eqref{compo}, so     Theorem~\ref{T1L}(ii) applies. The identity $\,\mathcal S(C) \wedge \mathcal S(X) = \mathcal R(C,W)$   is clear from \eqref{R}, \eqref{Rn}, as already noted in Subsection~\ref{SS:preview}. 

For the statement in part~(i) of the corollary, we appeal to part~(i) in Theorem~\ref{T1L} and to the notation introduced there, and obtain the representations 
\begin{align*}
	C( t) &= \Theta_{x_0}\big(\mathfrak{B}\left( t \,, W, X\right)\big)
 = \Theta_{x_0}\big(\mathfrak{B}\big( t \,, W  , \Theta_{x_0}\big(\Gamma +W \big)\big)\big) = \Theta_{x_0}\big(\Gamma( t)\big), \qquad 0 \leq t < {\cal R} (C,W);\\
\Theta_{C( t)} \big( W( t) \big) &= \Theta_{\Theta_{x_0} (\Gamma ( t))} \big( W(t) \big) \,=\,\Theta_{x_0} \big(\Gamma ( t) + W( t) \big) \,=\,   X( t), \qquad 0 \leq t < {\cal R} (C,W),
\end{align*}
the latter  on the strength of the composition property \eqref{compo}. These representations lead to the claims in part~(i) of the corollary; the claimed uniqueness for the OIE~\eqref{2} is argued as in Remark~\ref{R:ODE}.
\end{proof}

\begin{cor}[{\bf Barrow--Osgood conditions}]
\label{C:BO} We fix a semimartingale measure $\,\Prob \in \mathfrak{P}\,$
 and impose    the  Barrow--Osgood   conditions 
\begin{equation}
\label{Osg}
H_{x_0} (\ell +) =\, - \infty\,,\qquad  H_{x_0} (r -) =\,\infty \,.
\end{equation}
   Then, in the notation of Theorem~\ref{T1L} and Corollary~\ref{T1}, $\,\mathcal{R}(C,W) = \mathcal{S}(X)\,$ holds $\,\Prob$--almost surely.  Moreover, we have
\begin{align*}
	\big\{\mathcal{S}(X) = \infty\big\} \, = \,\big\{  [0, \infty) \ni t \,\longmapsto \,\mathfrak{B} \big(t, W, X\big) \text{ is real-valued}\,\big\}\,, \quad \text{mod.}~\Prob\,.
\end{align*}
\end{cor}
		
\begin{proof}
	Under the conditions of \eqref{Osg}  we have $\,\widetilde{\ell} (c)=-\infty$ and $\,\widetilde{r} (c)=\infty\,$ in   \eqref{eq:lc,rc}  for every $\, c \in \I$;  the function $\, \Theta_{x_0} (\cdot)\,$ is then defined on all of $\, \R\,$ and takes values in the interval $\, \mathcal{I} = ( \ell, r)$;  and   the domain of \eqref{eq:D} becomes the rectangle  $\, \mathcal{D}=\left\{(c,w)\,:\, c \in \I\,,\,\,  w \in \R \right\}= \I \times \R\,$.
In particular, we then have $\mathcal{S}(X) = {\cal S} (C)$ and thus $\mathcal{R}(C,W) = \mathcal{S}(X)$ by the definition in  \eqref{R}. 
By Theorem~\ref{T1L}(i)  we have the representation $\, 
	X(  t) = \Theta_{x_0} \big(\mathfrak{B}\big(  t\,, W , X  \big) + W(  t)\big)\,$,  
which then yields the stated set equality.
\end{proof}

\begin{rem} One might wonder how the stopping times $\mathcal{S}(C)$ and $\mathcal{S}(X)$ of Corollary~\ref{T1} relate to each other.  In general, without the Barrow-Osgood conditions~\eqref{Osg}, anything is possible, as we illustrate here with a brief example where   both events $\,\{ \mathcal{S}(X) <  \mathcal{S}(C) \}\,$ and $\, \{ \mathcal{S}(X) >  \mathcal{S}(C) \}\,$ have positive probabilities.  
  We consider $\, \mathcal{I} = (0, \infty)$,  $ \mathfrak{ B} (t, \cdot \,, \cdot) = t$ for all $t \geq 0$, and $\, \mathfrak{s} (x) =   x^2\,$ for all $\, x \in  \mathcal{I} \,$. Then
    $$
 H_c (x) \,=\, \frac{1}{\,c\,} - \frac{1}{\,x\,}, \quad  (c,x) \in (0, \infty)^{2} \qquad \text{and} \qquad  \Theta_c (w) \,=\, \left( \frac{1}{\,c\,\,} - w \right)^{-1}, \quad (c,w) \in \cal D 
  $$
with $\, {\cal D} = \big\{ (c,w) \in (0, \infty) \times  \mathbb{R} : - \infty < w <  1/c  \big\}$; in particular, $\, {\cal J} = ( - \infty,  \, 1/x_0 )\,$,  and the second of the Barrow-Osgood conditions~\eqref{Osg}   fails.   It follows   that 
\begin{align*}
 \Gamma (t) &= t, \quad 0 \leq t < \widetilde{{\cal S}}(\Gamma+W), \qquad
 C(t)= \left( \frac{1}{\,x_0\,} - t \right)^{-1}, \quad 0 \le t < \mathcal{S} (C)  = \frac{1}{\,x_0\,}, \\
 X(t) &= \Theta_{x_0} \big( \Gamma (t)+ W(t) \big) \,=\, \left( \frac{1}{\,x_0\,} - t  - W(t)\right)^{-1}, \qquad 0 \le t < \mathcal{S} (X) = \widetilde{{\cal S}}(t+W).
\end{align*}
Moreover, we have the representation $\mathcal{S} (X) =  \inf \{ t \ge 0 \,:\, t + W(t) = 1/x_0\}$ and it is clear that
both events $\,\{ \mathcal{S}(X) <  1/x_0 \}\,$ and $\, \{ \mathcal{S}(X) >   1/x_0 \}\,$ have positive probabilities.  
\qed
\end{rem}

\section{Examples}  
 \label{sec01}

 We view the term  corresponding to $\,\dx \mathfrak{B}(\cdot\,, W,X)\,$ in \eqref{1} as a sort of generalized or ``singular" drift that allows for both feedback effects $($the dependence on the past and present   of the ``state" process $X(\cdot))$  and feed-forward effects (the dependence on the past and present   of the ``input" process $W(\cdot))$.

\subsection{The case of no dependence on the  state process}  
\label{SS:5.1}

Let us consider mappings $\,\mathfrak{B}(\cdot\,, \cdot\,, \cdot)\,$ that do not depend on the state process  $X(\cdot)$, namely   
$$
\mathfrak{B}(t,\omega,\mathrm{x}) = B(t,\omega),  \qquad    (t,\omega , \mathrm{x} ) \in [0,\infty) \times C([0,\infty);  \R) \times C_a\big([0,\infty);  \overline{\I}\,\big)\,.
$$ 
Here  $\,B: [0,\infty) \times C([0,\infty);  \R) \rightarrow \R\,$ is some progressively measurable  mapping, such that $\,B(\cdot\,,  \omega )\,$ is continuous and of  finite first variation on compact intervals for all $\,\omega \in \Omega\,$.  It should be stressed that $\, B (\cdot\,, W)\,$ {\it need   not be absolutely continuous with respect to {Lebesgue}  measure.} 

We may define, for   some  bounded, measurable function $\beta: [0,\infty) \times \R \rightarrow \R$,  the progressively measurable functional 
\begin{align*} 
		B(t, \omega
	) = \int_0^t \beta \big(s,\omega(s)\big) \, \dx s\,, \qquad 0 \leq t < \infty\,, ~~~
	\omega   
	\in C \big([0,\infty); \R\big).
	\end{align*}
 We might be interested, for example,  in a continuous semimartingale $X(\cdot)$ that is positively drifted whenever the driving noise is positive; in such a case, we might consider, for example, $\,\beta( t, \omega) = \1_{\{\omega( t)>0\}}\,$ for all $\,(t,\omega) \in [0,\infty) \times C ([0,\infty); \R)$.   Alternatively, we may take   $\, B(T, \omega) = \int^{T}_0 \beta (t)\, \mathrm{d} t$, $\,(T, \omega) \in [0, \infty) \times \Omega\,$ for some integrable function  $\beta : [0, \infty) \ra \R$.   

In this setting, the SIE~\eqref{1}     takes the form 
\begin{align} 
\label{eq:ex1.3} 
X(\cdot) &= x_0 + \int_0^{\,\cdot} \mathfrak{s} (  X(t))\, \big[\, \dx W(t) +   \dx B (t,W)  \, \big] - \int_{\I} L^X (\cdot\, , \xi)  \, \mathfrak{s}  (\xi)\, \dx \frac{1}{\mathfrak{s}  (\xi) } \,,
\end{align}
and the corresponding OFE~\eqref{2L} and OIE~\eqref{2} become respectively 
  $\, \Gamma (\cdot) = {B} (\cdot \,, W)\,$ and 
$\,C( \cdot) = x_0 + \int_0^{\, \cdot} \mathfrak{s} \big( C(t) \big)\, \dx B (t,  W(t)   )\,$.  In particular, under any semimartingale measure $\, \Prob \in \mathfrak{P}\,$,  the solutions of \eqref{7}  and \eqref{eq:ex1.3} are then expressed  as  
\begin{align} 
\label{eq: sp independent 1}
C(T) &= \Theta_{x_0} \big(B(T, W)   \big) , \qquad 0 \le T <   \widetilde{\cal S}\big( B(\cdot \, , W)  \big) \equiv {\cal S} (C)  \nonumber\\
 X( T) &= \Theta_{x_0} \big( \mathfrak{B} (T, W) + W(T) \big), \qquad 0 \le T <\widetilde{\cal S}\big(B(\cdot \, , W) + W\big) \equiv {\cal S} (X)\, . 
\end{align} 
 Under the  conditions   \eqref{Osg}, there are no explosions in the present context; i.e., $\mathcal{S}(X) = \mathcal{S}(C) = \infty$.

 The state process $X(\cdot)$ in \eqref{eq: sp independent 1} is  adapted not only to the right-continuous version $\,\mathbb{F}
\,$ of the pure filtration $\mathbb{F}^W$,  but also to this pure filtration itself. And if the function $B(\cdot, \cdot)$ does not depend on the second argument -- that is, if $B(\cdot\,, \omega) = B(\cdot)\,$    is   equal to a given measurable  function of finite first variation on compact subsets of $\,[0, \infty)\,$, for every $\, \omega \in \Omega\,$ -- then for each $t \in [0,\infty)$ the random variables $X(t)$ and $W(t)$ are actually bijections of each other; to wit, $\, \sigma (X (t)) = \sigma(W (t))\,$ holds. Finally, we  note that in the  trivial case $B (\cdot\,, \cdot\,, \cdot) \equiv 0\,$ the solution in \eqref{eq: sp independent 1} simplifies further to $\,X( \cdot) = \Theta_{x_0}  ( W ( \cdot) )\, $.

The next example illustrates that  the above arguments   can be generalized somewhat. 

 \begin{example}
	In the notation of this subsection, let  
	$\,A: \R \rightarrow (0,\infty)\,$ be a measurable function such that $1/ A(\cdot)$ is integrable on compact subsets of $\R$. Moreover, we shall consider  a continuous mapping 
	$\, t \mapsto B( t, \omega)$ of finite first variation on compact intervals for all $\omega \in \Omega\,$.  Let us  fix
$$
\mathfrak{B}(T,\omega,\mathrm{x}) = \1_{ \{\varrhob(\omega,\mathrm{x})>T \}} \int_0^{T } A\big(H_{x_0}(\mathrm{x}(t))-\omega(t)\big)\, \dx B(t,\omega) 
$$ 
for all $\, (T, \omega , \mathrm{x} ) \in [0,\infty) \times  C([0,\infty);  \R) \times C_a \big([0,\infty);  \overline{\I}\, \big)\,$, with 
$$
\varrhob (\omega,\mathrm{x}) \,:=\, \inf \left\{ T \ge 0 \,:\, \int_0^{T } A\big(H_{x_0}(\mathrm{x}(t))-\omega(t)\big)\, \dx | B |(t,\omega) = \infty \right\}. 
$$
Then  the corresponding SIE~\eqref{1} can be written as
\begin{align} 
\label{1:ex5.1} 
	X(\cdot)  = x_0 + \int_0^{\,\cdot} \mathfrak{s} (  X(t))  \big[  \dx W(t) + A\big(H_{x_0}(\mathrm{x}(t))-W(t)\big)   \dx B (t,W)    \big] - \int_{\I} L^X (\cdot\, , \xi)  \, \mathfrak{s}  (\xi)\, \dx \frac{1}{ \mathfrak{s}  (\xi) } 
\end{align}
and the corresponding OFE~\eqref{7L} as
\begin{align} 
\label{1:ex5.3} 
\Gamma(\cdot) =  \int_0^{\, \cdot}  A\big(\Gamma(t)\big)  \,\dx B(t,W).
\end{align}

We define now the functions $\,\mathbf{ H}_0(\cdot)\,$ and $     {\bm \Theta}  _0(\cdot)$ in the same way as   $H_0(\cdot)$ and $\Theta_0(\cdot)$  but   with $\mathfrak{s}(\cdot)$ replaced by $A(\cdot)$. The arguments in Remark~\ref{R:ODE} show that the unique solution of the equation \eqref{1:ex5.3} is  
 \begin{align*} 
	\Gamma(t) \,=\, {\bm \Theta}_0 \big(B( t,W)\big)\,, \quad 0 \le t < \widetilde{{\bm S}} \big(B(\cdot \,,W)\big)
\end{align*}
with 
\begin{align*}
\widetilde{{\bm S}} \big(B(T,\omega)\big)\, :=\, \inf \left\{ t \ge 0\,:\, B \big(t, \omega \big)\notin \left( - 
	\int_{-\infty}^0 {\dx z \over \, A(z)\,}\,, \, \,
	 \int_0^\infty {\dx z \over \,  A(z)\,} 
	 \right) \right\}  
\end{align*}
in the manner of the stop-rule in \eqref{eq:S2}; and that the SIE~\eqref{1:ex5.1}  has then a unique $\mathbb{F}^W$--adapted solution  under each probability measure $\,\Prob \in \mathfrak{P}\,$, namely 
\begin{align} 
\label{1:ex5.7}
X(t) \,=\,{ \Theta}_{ x_0 } \big( {\bm \Theta}_{ 0 } ( B( t,W) ) + W (t) \big)\,, \qquad 0 \le t < \widetilde{\mathcal{S}} \big( {\bm \Theta}_{  0 } ( B( \cdot \,,W) ) + W   \big) = \mathcal{S} (X).
\end{align}

For instance, let us consider the case $\,\I = (0,\infty)\,$, $\,x_0=1$, and $\,\mathfrak{s}(x) = x\,$ for all $\,x \in \I\,$. Then we have $H_{x_0} (x) = \log ( x  ) \,$ for all $x \in \mathcal{I} $,  and the equation of \eqref{1:ex5.1} simplifies to
\begin{align} 
\label{1:ex5.2} 
	X(\cdot) &= 1 + \int_0^{\cdot} X(t) \Big[\, \dx W(t) +   
	A\big(\log(X(t))-W(t)\big) \,\dx B (t,W)    
	+  \frac{1}{\,2\,}  \,
	\dx \langle W\rangle(t) \, \Big]\,.
\end{align}
\noindent
This SIE has then a unique $\mathbb{F}^W$--adapted solution  under each probability measure $\,\Prob \in \mathfrak{P}\,$, given by \eqref{1:ex5.7} 
as
\begin{align} 
\label{1:ex5.7b}
X(t) \,= \exp\big({\bm \Theta}_{ 0 } ( B( t,W) ) + W (t) \big)\,, \qquad 0 \le t < \widetilde{{\bm S}} \big(B(\cdot \,,W)\big)  = \mathcal{S} (X).
\end{align}
More specifically, let us consider the case $A(x) = \exp(-x)$ for all $\,x \in \R$. Then we have $\,{\bm \Theta}_{ 0 }(y) = \log(1+y)\,$ for all $\,y \in (-1, \infty)$, and the SIE~\eqref{1:ex5.2} simplifies to 
\begin{align*} 
	X(\cdot) &= 1 + \int_0^{\cdot} X(t)  \dx W(t) + \int_0^\cdot \exp\big(W(t)\big)\, \dx B(t,W) + \frac{1}{2}  \int_0^\cdot X(t) \, \dx \langle W\rangle(t) \,;
\end{align*}
from \eqref{1:ex5.7b}, the unique solution of this stochastic integral equation  is $\, 
X(t) = \big(1+B(t,W)\big) \exp\big( W (t) \big)\,$, $\,0 \le t < \infty\,$,
and it is easy to check that this expression indeed  solves the equation. \qed
\end{example}

 \subsection{Absolutely continuous drifts} \label{SS:Ex2}
Another very important example for the term $\,\dx \mathfrak{B}(\cdot,W, X)$ involves a measurable function $\,\mathfrak{b}: [0,\infty) \times \R \times \I \rightarrow \R\,$    such that, for all $ \,(T,K) \in (0, \infty)^2\,$, the functions
\begin{equation*}
\overline{\,\mathfrak{b}}_K(\cdot) :=  \sup_{(t,w) \in [0 , T] \times   [-K,K]} \big|  \mathfrak{b} ( t, w, \cdot) \big| ~~\text{ are  integrable on compact subsets of } \,\I \,; 
\end{equation*}
see \citet{Engelbert1991}. For any given  $\,(\omega , \mathrm{x} )  \in C ([0,\infty); \R) \times C_a \big([0,\infty); \overline{\I}\,\big)$,  we   define   
\begin{equation*}
	\mathfrak{B}(T,\omega,\mathrm{x}) :=  \1_{ \{\varrhob(\omega,\mathrm{x})>T\}}  \int_0^T \mathfrak{b}\big(t,\omega(t), \mathrm{x }(t)\big) \, \dx t
\end{equation*}
for  all $T \geq 0$ along with the stop-rule   
$$
\varrhob(\omega, \mathrm{x}):= \inf \left\{ T \ge 0\,: \,\int_0^T  \big| \mathfrak{b}(t,\omega(t),\mathrm{x}(t)) \big|  \,\dx t \,=\,   \infty \,\right\} \,.
$$    
 The SIE~\eqref{1}  takes then the form
\begin{align}
\label{1b}
X(\cdot) \,=\, x_0 + \int_0^{\,\cdot} \mathfrak{s} (  X(t))\, \big[ \,\dx W(t) + \mathfrak{b} \big(t,W(t),X(t)\big) \,\dx t \,\big] - \int_{\I} L^X (\cdot\, , \xi) \, \mathfrak{s}  (\xi)\, \dx \frac{1}{\mathfrak{s}  (\xi) } \, ;
\end{align}
and when  $\mathfrak{b}(\cdot\,, \cdot\,, \cdot)$ does not depend on the second argument, this equation simplifies further  to the SIE~\eqref{1c}.  In the context of this example, \eqref{2L} becomes an OIE of the form  
 \begin{equation}
 \label{2LL}
\Gamma (T) \,=\,  \int_0^T \mathfrak{b} \big(t,   W, \Theta_{x_0} (  \Gamma (t) + W  (t)  ) \big)\, \dx t  \,, \qquad 0 \le T <  \widetilde{\cal S} (\Gamma +W) = {\cal S} (X).
\end{equation}
On the other hand, the   OIE~\eqref{7} corresponding to the SIE~\eqref{1b} takes the form
\begin{equation}
C(\cdot) \, =  \,x_0 + \int_0^{\,\cdot} \mathfrak{s} \big( C(t) \big) \, \mathfrak{b}\left(t,  \omega(t), \Theta_{C(t)} ( \omega(t)   ) \right)  \dx t    \,. \label{77} 
\end{equation}

\begin{rem}   \label{BO_Bound} 
Under the Barrow-Osgood conditions of \eqref{Osg}  we have $\mathcal{S}(C) = \mathcal{S}(X)$ by Corollary~\ref{C:BO} and also $\mathcal{S}(X) = \varrhob(W,X)\,$, since those conditions imply  $$\,\mathcal{S}(X) =\widetilde{\cal S} (\Gamma +W) = \widetilde{\cal S} (\Gamma) = \varrhob(W,X) .$$ 
 In particular, if the drift function $\,\mathfrak{b}(\cdot\,,  \cdot\,,  \cdot)\,$ is bounded, then all the stopping times in the above display are infinite. \qed
\end{rem}

\begin{example}  [{{\it A Counterexample}}] 
We cannot expect the SIE~\eqref{1}, or for that matter the the OIE~\eqref{7}, to admit $\,\mathbb{F}
$--adapted solutions for a general  progressively measurable  functional $\,\mathfrak{ B} (\cdot\,, \cdot\,, \cdot)    $.  For instance, take $\, \mathbb{P} = \mathbb{P}_*\,$ to be Wiener measure, take $\,\mathfrak{s} (\cdot )   \equiv 1\, $, and consider 
$$
\mathfrak{ B} (T, \omega, \mathrm{x})\,=\, \int_0^T \mathrm{b}  ( t,   \mathrm{x})\, \dx t \,, \qquad (T,\omega, \mathrm{x}) \in [0, \infty) \times C\big([0,\infty); \R \big) \times C \big([0,\infty); \R \big)
$$
for the  bounded drift
$$
\mathrm{b} ( t,   \mathrm{x}\,)= \left\{ \frac{\, \mathrm{x}(t_k) - \mathrm{x}(t_{k-1})\,}{  t_k  -  t_{k-1}} \right\}, \quad  t_k  <t\le  t_{k+1}\,; \qquad \mathrm{b} \big( t,   \mathrm{x}\,\big)=0 \quad \text{for}~~ t=0\,,~t>1 
$$
of \citet{Tsirelson}. Here $\{ \xi\}$ stands for the fractional part 
of the number $\, \xi \in \R\,$, and $\big( t_k\big)_{k \in - \N}\,$ is a strictly increasing sequence of numbers with $\, t_0=1\,$, with $\, 0<t_k <1\,$ for $\, k<0\,$, and with $\, \lim_{k \downarrow - \infty} t_k =0\,$. It was shown in the landmark paper of \citet{Tsirelson} (see also pages~195-197 of \citet{Ikeda_Watanabe} or pages~392-393 of \citet{RY}) that the resulting SIE
$$
X(\cdot)\,=\, x_0 + \int_0^{\, \cdot} \mathrm{b} ( t, X  ) \, \mathrm{d} t + W(\cdot)
$$   
in \eqref{1}, driven by the $\, \Prob_*$--Brownian motion $\, W(\cdot)\,$,  admits a   weak solution which is unique in distribution, but {\it no strong solution}; see also the deep work of \citet{Benes1977, Benes1978} for far-reaching generalizations and interpretations of Tsirelson's result. As a result, the OIE
$$
C(\cdot)\,=\, x_0 + \int_0^{\, \cdot} \mathrm{b} \big( t, W  + C   \big) \, \mathrm{d} t  
$$ 
of \eqref{7} cannot possibly admit an $\,\mathbb{F}$--adapted solution in this case.   \qed
\end{example}

\subsection{The time-homogenous case}
 \label{sec3.1}

We consider a measurable function   $\,\mathfrak{b}: \I \rightarrow \R\,$ which is integrable on compact subsets of $\I\,$, as well as a  signed measure $\, {\bm \mu}\,$ on the Borel sigma algebra $\, {\cal B} ( \I)$ which is finite on compact subsets of $\,\I\,$. As in Subsection~\ref{SS:Ex2}, we then introduce the progressively measurable mapping 
\begin{equation*}
	\mathfrak{B}(T,\omega,\mathrm{x}) \equiv \mathfrak{B}(T, \mathrm{x}) :=  \1_{ \{\varrhob( \mathrm{x})>T\}} \left(\int_0^{T} \mathfrak{b} ( \mathrm{x} (t) )  \dx t + \int_\I L^\mathrm{x}(T, \xi ) \, \frac{\,{\bm \mu} (\dx  \xi)\,}{\mathfrak{s}(\xi)} \right), \qquad  T  \in [0, 	\infty)  
\end{equation*}
as well as the stop-rule 
$$\,    \varrhob(\omega, \mathrm{x})   \equiv 
\varrhob(  \mathrm{x}) := \inf \left\{ T \ge 0\,: \,\int_0^T  \left|\mathfrak{b}(\mathrm{x}(t))\right|  \,\dx t +  \int_\I L^\mathrm{x}(T \,, \xi ) \, \frac{\,{| \bm \mu |} (\dx  \xi)\,}{\mathfrak{s}(\xi)}  \, = \infty \,\right\}  
$$    
for all  $\,(\omega , \mathrm{x}  ) \in C([0,\infty); \R) \times C_a \big([0,\infty); \overline{\I}\, \big)$.  
With this choice of drift, the SIE~\eqref{1} can be written as
\begin{align}
X(\cdot) 	&\,= \,x_0 + \int_0^{\,\cdot} \mathfrak{s} (  X(t)) \, \big[ \,\dx W(t) + \mathfrak{b} (  X(t)) \,\dx t \, \big] + \int_{\I} L^X (\cdot\, , \xi) \,  \Big[\, {\bm \mu} ( \dx  \xi) - \, \mathfrak{s}  (\xi)\, \dx \frac{1}{\mathfrak{s}  (\xi) } \, \Big]\,; \label{eq:1.3b} 
\end{align}
whereas the corresponding OFE~\eqref{2L} and OIE~\eqref{7} take respectively the form
\begin{align}
\label{77c}
\Gamma (\cdot) &=   \int_0^{\,\cdot}      \,\mathfrak{b}\big(  \Theta_{x_0} ( \Gamma(t)  + W(t) ) \big) \, \dx t +   \int_\I L^{ \Theta_{x_0} ( \Gamma + W )} \left(\cdot\, , \xi \right) \, \frac{\, {\bm \mu} (\dx  (\xi)\,}{\mathfrak{s}(\xi)}  ;
\\
\label{77b}
C(\cdot) &=  x_0 +  \int_0^{\,\cdot} \mathfrak{s} \big( C(t) \big) \left[\,\mathfrak{b}\left(  \Theta_{C(t)} ( \omega(t)   ) \right)  \dx t +   \int_\I L^{ \Theta_{C(\cdot)} ( \omega(\cdot) )} \left(\dx t, \xi \right) \, \frac{\, {\bm \mu} (\dx  (\xi)\,}{\mathfrak{s}(\xi)}   \, \right]. 
\end{align}
 We also note that the special case $\, {\bm \mu} \equiv 0\,$ leads to  the time-homogeneous version  
\begin{equation*}
X(\cdot) \,=\, x_0 + \int_0^{\,\cdot} \mathfrak{s} (  X(t))\, \big[ \,\dx W(t) + \mathfrak{b} (X(t))\, \dx t \, \big] - \int_\I L^X (\,\cdot\, , \xi)   \, \mathfrak{s}  (\xi)\, \dx \, \frac{1}{\mathfrak{s}  (\xi) } \,.  
\end{equation*}
of \eqref{1c}. If we choose the measure $\, {\bm \mu}\,$ so that  $\, {\bm \mu}([a,b)) =   \int_{[a,b)} \, \mathfrak{s}  (\xi)\, \dx \big( 1 / \mathfrak{s}  (\xi)  \big) \,$ holds for all $\,(a, b) \in \I^{\,2}\,$ with $\,a<b\,$, then the local time term in \eqref{eq:1.3b} disappears  entirely.

\subsubsection*{The time-homogenous case under Wiener measure}
 \label{sec3.1W}

 Let us  consider  next  {\it under the Wiener measure $\, \Prob_*\,$}   
   the SIE~\eqref{eq:1.3b},  now written in the more ``canonical" form 
\begin{equation*}
X(\cdot) \,=\, x_0 + \int_0^{\,\cdot} \mathfrak{s} (  X(t))\, \dx W(t) + 
 \int_\I L^X (\,\cdot\, , \xi) \, {\bm \nu} (\dx  \xi)\, ;
\end{equation*}
 here $\, {\bm \nu}\,$ is the  measure on the Borel sigma algebra of $\I$, given by
\begin{equation}
\label{meas:nu}
{\bm \nu} \big( [a,b) \big) \,=\, {\bm \mu} \big( [a,b) \big) - \int_{[a,b)} \, \mathfrak{s}  (\xi)\, \dx \frac{1}{ \mathfrak{s}  (\xi)  } + \,2 \int_a^b { \mathfrak{b}  (\xi) \over \,\mathfrak{s}  (\xi)\,}\,\dx \xi\,, \qquad \ell < a < b <r\,.
\end{equation}

   \begin{thm}  \label{T:5.2} 
   In the context of this subsection, suppose that  the signed measure $\,{\bm  \nu}\,$ of \eqref{meas:nu}    satisfies
$$
{\bm \nu} \big( \{ x\} \big) \, < \, 1\,, \qquad x \in \I\,.
$$
Suppose also that there exist an increasing function $\, f: \I \ra \R\,$,  a nonnegative, measurable function $\, g: \R \ra [0, \infty)\,$, and  a real constant $c >0 $,  such that we have 
$$ 
\int_{-\varepsilon}^{\varepsilon}  \, \frac{\, \dx y \,}{ g(y)}\, =\, \infty, \qquad \varepsilon > 0\,,
$$ 
as well as 
$$
\big| \mathfrak{s} (\xi + y) - \mathfrak{s} (\xi) \big|^2 \, \le \, \frac{\,g(y)\,}{|y|}\, \big| f(\xi + y) - f(\xi) \big|    
$$
for all $\,\xi \in \I$ and $\,y \in (-c,c) \setminus \{ 0\}\,$ with $\xi +y \in \I\,$.

 Then under the Wiener measure $\, \Prob_*\,$, the SIE   of  
 \eqref{eq:1.3b}  has a pathwise unique, $\mathbb{F}$--adapted  
 solution $\, X(\cdot)\,$. Therefore, on account of Theorem~\ref{T1L},    and again under the Wiener measure $\, \Prob_*\,$, the  OFE~\eqref{77c} has also a unique $\mathbb{F}$--adapted solution $\Gamma(\cdot)\,;$   and these solutions are related via   the   evaluation of \eqref{8L}, namely  
$X( t) = \Theta_{x_0} ( \Gamma(t)+W( t)  )\,$ for $\,  0 \le t <\mathcal{S}(X) =  \widetilde{\cal S}(\Gamma + W)$.
\end{thm}

The   first claim of Theorem~\ref{T:5.2}   is proved as in Theorem~4.48 in \citet{Engelbert_Schmidt_1991}; see also   \citet{ LeGall1984}, \citet{BarlowPerkins}, \citet{Engelbert_Schmidt_lecture},  
and \citet{BleiEngelbert2012, BleiEngelbert2013}. 
The argument proceeds by  the familiar \citet{Zvonkin1974} method of {\it removal of drift;}  \citet{StroockYor1981}, \citet{LeGall1984}, and \citet{Engelbert_Schmidt_lecture} contain early usage of this technique in the context of stochastic integral equations with generalized drifts.  In these works  the filtration is   
augmented by the $\Prob_*$--nullsets; however, there always exists a $\Prob_*$--indistinguishable modification $X(\cdot)$ of the solution process that is $\mathbb{F}$--adapted (see Remark~I.1.37 in \citet{JacodS}).

This reduction to a diffusion in natural scale, along with     the classical {Feller}  test,  leads to   necessary and sufficient conditions for the absence of explosions $\,\Prob_*  \big(\mathcal{S} (X) = \infty\big)=1\,$ in the spirit of \citet{MU_integral}, \citet{Karatzas_Ruf_2013}; the straightforward details are   left to the diligent reader.

\subsection{A close relative of the Skew Brownian Motion}  \label{SS:Ex4}
For two given real numbers $\, \rho >0\,,\, \sigma >0\,$, let us consider the SIE
\begin{equation}
\label{skew-like}
X(\cdot) \,=\, \int_0^{\, \cdot} \Big( \rho \, \1_{ (- \infty, 0]} (X(t)) + \sigma \, \1_{ (0,  \infty)} (X(t))\Big) \dx W(t) \,+\, \frac{\, \sigma - \rho\, }{\sigma } \, L^X (\cdot\,, 0)\,.
\end{equation}
    This corresponds to the equation \eqref{1} with $\,\mathfrak{B}(\cdot\,, \cdot\,, \cdot) \equiv 0\,$, state space $\, \I = \mathbb{R}\,$, initial condition $\,x_0=0$ and dispersion function $\,  \mathfrak{s} = \rho \, \1_{ (- \infty, 0]} + \sigma \, \1_{ (0,  \infty)}\,$, thus 
    $$
 H_0 (x) = \frac{\,x\,}{\rho}\,     \1_{ (- \infty, 0]} (x) + \frac{\,x\,}{\sigma}\,     \1_{ (0, \infty)} (x)\,, \qquad \Theta_0 (w) =   \rho \, w\,     \1_{ (- \infty, 0]} (w) +  \sigma \,   w\,  \1_{ (0, \infty)} (w) 
        $$
for the function of \eqref{3} and its inverse.  The Barrow-Osgood conditions \eqref{Osg}    are obviously satisfied here, explosions are non-existent, whereas Theorem~\ref{T1L} or Corollary~\ref{T1} imply that 
\begin{equation}
\label{new_skew}
 X(t) \,=\, \Theta_0 \big( W(t)\big) \,=\, \sigma\, W^+ (t) - \rho\, W^- (t)\,, \qquad 0 \le t < \infty 
\end{equation}
is the unique solution of \eqref{skew-like}.    Indeed, it can be checked by fairly straightforward application of the It\^o-Tanaka formula,  that the process of \eqref{new_skew}  satisfies SIE~\eqref{skew-like} under  any semimartingale measure $\, \mathbb{P} \in  \mathfrak{P}\,$; and conversely, that every solution of this equation has to be given by the    expression in \eqref{new_skew}.  
    
   Suppose now that the canonical process $\, W(\cdot)\,$ is Skew Brownian motion with parameter $\, \alpha \in (0,1) $, to wit,  that   the process $\, V(\cdot) \equiv W(\cdot) - \big( ( 2 \alpha - 1) / \alpha\big) L^W(\cdot\,, 0)\,$ is standard Brownian motion, under the probability measure $\, \mathbb{P} \in  \mathfrak{P}\,$ (cf.$\,$\citet{HarrisonShepp}). Then it can be checked, by considerations similar to those   in Remark~\ref{insidious},  that  the SIE~\eqref{skew-like} takes the equivalent form
 \begin{equation*}
X(\cdot) \,=\, \int_0^{\, \cdot} \Big( \rho \, \1_{ (- \infty, 0]} (X(t)) + \sigma \, \1_{ (0,  \infty)} (X(t))\Big) \dx V(t) \,+\, \left( 1 - \frac{\, (1-\alpha)\, \rho\, }{\alpha \, \sigma } \right) L^X (\cdot\,, 0)\,.
\end{equation*}
  Such  equations    have been studied before, for example by \citet{Ouknine1991} and \citet{Lejay2006}.

  \section{A comparison result}
  \label{comp}
  
Let us place ourselves  again in the context of Subsection~\ref{SS:Ex2} with  the function  $\, \mathfrak{b} : [0, \infty) \times \R \times \I \rightarrow \R\,$ {\it continuous,} and fix an arbitrary  semimartingale measure $\, \Prob \in \mathfrak{P}\,$. Then, in terms of  the continuous, real-valued function  
  \begin{equation}
\label{cue}
  G (t, w, \gamma)\, := \,    \mathfrak{b}  \big(t,w, \Theta_{x_0} (  \gamma + w) \big)\, ;\quad t \in [0, \infty), ~~( \gamma , w) \in {\cal E}:= \big\{(\gamma, w) \in \R^2: \gamma + w \in \J \big\}
  \end{equation} 
we can write the OFE~\eqref{2LL} in the slightly more compact form 
  \begin{equation}
\label{CQW}
\Gamma (\cdot)\, =\,   \int_0^{\, \cdot} G \big(t, W(t), \Gamma(t)\big)  \dx t\,.
  \end{equation} 

From Theorem~III.2.1 in \citet{Hartman1982}, we know that this equation has a {\it maximal}  solution $\, \overline{\Gamma}  (\cdot)\,$, defined up until the  time $\, \widetilde{{\cal S}} (\overline{\Gamma}+ W )\,$. Assuming that this solution $\, \overline{\Gamma}  (\cdot)\,$ is $\,\mathbb{F}$--adapted, we observe --- on the strength of Theorem~\ref{T1L} and of the strict increase of the mapping $\,   \Theta_{x_0} (\cdot)\,$ (see also \eqref{5}) --- that the corresponding $\,\mathbb{F}
$--adapted process 
$$\, \overline{X}  (t) \,:=\, \Theta_{x_0} \big(\overline \Gamma (t) +W(t)\big)\,, \quad \,\, 0 \le t < \widetilde{{\cal S}} \big(\overline{\Gamma}+ W \big) = {\cal S} (\overline{X}) $$ from \eqref{8L}, \eqref{eq:S} is the {\it maximal  solution} on the filtered probability   space $\, ( \Omega, {\cal F}, \Prob),\,  \mathbb{F} =  \{ \mathcal{F}
(t) \}_{0 \le t < \infty}\,$  of the SIE  \eqref{1b}, namely, 
$$
X(\cdot) \,=\, x_0 + \int_0^{\,\cdot} \mathfrak{s} (  X(t))\, \big[ \,\dx W(t) + \mathfrak{b} \big(t,W(t),X(t)\big) \,\dx t \,\big] - \int_{\I} L^X (\cdot\, , \xi) \,  \mathfrak{s}  (\xi)\, \dx \frac{ 1} { \mathfrak{s}  (\xi)  }\,.
$$

We fix now a number $\,\widehat{x}_0 \in (\ell, x_0]\,$ and   consider   yet another continuous function $\,\widehat{\mathfrak{b}} : [0, \infty) \times \R \times \I \rightarrow \R\,$ satisfying the pointwise comparison 
\begin{equation}
\label{bb}
\widehat{\mathfrak{b}} (t,w,x)  \,\le \, \mathfrak{b}  (t,w,x)   \, ,  \qquad (t,w,x) \in [0, \infty) \times \R \times \I\,,
\end{equation}
 thus also  the comparison 
$$\widehat{\mathfrak{b}}     \big(t,w, \Theta_{x_0} (  \gamma + w) \big) =: \widehat{G} (t, w, \gamma) \, \le G (t, w, \gamma), \qquad t \in [0, \infty), ~~( \gamma , w) \in  {\cal E} 
   $$  
with the notation of \eqref{cue}.   Then we know from Theorem~\ref{T1L} that {\it any} $\,\mathbb{F}$--adapted process $\widehat{X}  (\cdot) \,$ satisfying, on the filtered probability  space $\, ( \Omega, {\cal F}, \Prob),\,  \mathbb{F}=  \{ \mathcal{F} (t) \}_{0 \le t < \infty}\,$, the equation 
\begin{equation}
\label{1bb}
\widehat{X}(\cdot) \,=\, \widehat{x}_0 + \int_0^{\,\cdot} \mathfrak{s} (  \widehat{X}(t))\, \left[ \,\dx W(t) + \widehat{\mathfrak{b}} \big(t,W(t),\widehat{X}(t) \big) \,\dx t \,\right] - \int_{\I} L^{\widehat{X}} (\cdot\, , \xi) \,    \mathfrak{s}  (\xi)\, \dx \frac{1 }{ \mathfrak{s}  (\xi) } ,
\end{equation}
  can be cast in the manner of \eqref{8L}, \eqref{eq:S} as 
$$\,
\widehat{X}( t) =  \Theta_{x_0} \big(\,\widehat{\Gamma} (t) +W(t)\big)\,,  \qquad 0 \le t < \mathcal{S} (\widehat{X}) = \widetilde{{\cal S}} \big(\widehat{\Gamma} + W \big)\, .
$$ 
Here the $\,\mathbb{F} $--adapted process $\widehat{\Gamma}  (\cdot) \,$ satisfies, up until the stopping  time $\, \mathcal{S} (\widehat{X}) =\widetilde{ {\cal S}} \big(\widehat{\Gamma}+ W \big)\,$, the analogue of the OIE~\eqref{CQW}, namely
$$
\widehat{\Gamma }(\cdot)\, =\,   \int_0^{\, \cdot} \widehat{G} \big(t, W(t), \widehat{\Gamma}(t)\big) \, \dx t\,.
$$
Corollary~III.4.2 in \citet{Hartman1982} asserts now that the comparison $\widehat{\Gamma}  (\cdot) \le \overline{\Gamma}(\cdot) \,$ holds on the interval $[0, \widetilde{{\cal S}} (\widehat{\Gamma}+ W )\wedge  \widetilde{{\cal S}} ( \overline{\Gamma} + W ) )\,$  and, from the strict increase of the mapping $\,   \Theta_{x_0} (\cdot)\,$ once again,  we deduce the {\it ``comparison for SIEs"} result 
$$
\widehat{X}  ( t) \le \overline{X}  ( t)\,, \qquad 0 \le t < \mathcal{S} (\widehat{X})  \wedge  \mathcal{S} (\overline{X})\,.
$$
This compares the maximal solution $\, \overline{X}  (\cdot)\,$ of the SIE~\eqref{1b} to an arbitrary solution $\widehat{X}  (\cdot)\,$ of the SIE~\eqref{1bb}, under the conditions $\widehat{x}_0 \le x_0\,$ and \eqref{bb}.

  \section{Continuity of the input-output map}
  \label{cont}

Once it has been established that the equation \eqref{1b} can be solved pathwise under appropriate conditions, it is important from the point of view of modeling and approximation to know whether the progressively measurable mapping $\,\mathfrak{X}: [0,\infty) \times C([0,\infty); \R) \rightarrow \overline \I\,$ that realizes its solution  $\, X(\cdot) = \mathfrak{X} (\cdot\,, W)\,$ in terms of the canonical process $\,W(\cdot)\,$ (the ``input" to this equation) is actually a {\it continuous} functional. 

The first result of this type for classical SDEs was established by \citet{Wong_Zakai_a, Wong_Zakai_b}; similar results with simpler proofs were obtained by \citet{Doss1977} and \citet{Sussmann1978}. Wong-Zakai-type approximations have been the subject of intense investigation.   
Some pointers to the relevant literature are provided in \citet{McShane1975}, \citet{Protter1977}, \citet{Marcus1981}, 
\citet{Ikeda_Watanabe}, \citet{KurtzPardouxProtter}, \citet{BassHamblyLyons}, \citet{AidaSasaki},  \citet{DaPelo2013}, and \citet{Zhang_2013}.

We are now ready to state the main result of the present section, and two important corollaries.    In order to simplify the exposition, we shall place ourselves in the context of Subsection~\ref{SS:Ex2}, impose the Barrow-Osgood conditions~\eqref{Osg}, and assume that the drift function $\mathfrak{b}(\cdot\,, \cdot\,, \cdot)$ is bounded. In light of the Remark~\ref{BO_Bound}, the SIE~\eqref{1b} is then  free of explosions. 

\begin{thm}[{\bf Continuity of  the input-output map}] 
\label{WZ} 
	In the context of Subsection~\ref{SS:Ex2} and under the Barrow-Osgood conditions~\eqref{Osg}, we assume that the drift function $\mathfrak{b}(\cdot\,, \cdot\,, \cdot)$ of the SIE~\eqref{1b}  is bounded and satisfies, for each given  $\,n \in \N\,$,  the following conditions:
\begin{enumerate}[(i)]
	\item   the function   $\R \ni w \mapsto  \mathfrak{b}(t,w,\xi)$ is continuous for all $\,(t,\xi) \in [0, \infty) \times \I\,;$  and 
	\item for all $(t,w,\xi_1, \xi_2) \in [0,n] \times [-n,n] \times (\ell_n, r_n)^2$, the   	local Lipschitz condition 
\begin{align*}
	\big|\mathfrak{b}(t,w, \xi_1) - \mathfrak{b}(t,w,\xi_2)\big| \, \leq  \, L_n\, |\xi_1 - \xi_2| 
\end{align*}
is satisfied, where   the  constant $L_n<\infty$   depends only on the integer $n$.
\end{enumerate}
Then the following statements hold: 
\begin{enumerate}
	\item  For each path $\,\omega  \in C([0,\infty);\R),$ the OIE~\eqref{2LL} has a unique  solution $\Gamma_\omega(\cdot)$. This solution is progressively measurable and satisfies $\Gamma_\omega(\cdot) \in \R$.
	\item If   $\,\big\{ \omega_k (\cdot)\big\}_{k \in \N}\,$ is a sequence of continuous paths in $\,C([0,\infty);\R)$ such that 
\begin{equation}
\label{app1} 
 \lim_{k \uparrow \infty} \sup_{t \in [0,n]} \big|\omega(t) - \omega_k(t)\big| = 0 \,,\qquad  n \in \N  
\end{equation}
holds for some $\omega (\cdot) \in C([0,\infty);\R)$, then with 
\begin{equation*}
\mathbf{ x} (\cdot) \, :=  \, \Theta_{x_0} \big(\Gamma_\omega(\cdot) + \omega(\cdot)\big) \qquad \text{and} \qquad \mathbf{ x}_k (\cdot) \, := \,  \Theta_{x_0}\big(\Gamma_{\omega_k}(\cdot) + \omega_k(\cdot)\big)\,,
\end{equation*} 
we have 
\begin{equation}
\label{app2}
\lim_{k \uparrow \infty}\, \sup_{t \in [0,n]} \big| \mathbf{ x} ( t)  -  \mathbf{ x}_k ( t) \big|  = 0 , \qquad n \in \N\,.  
\end{equation}
 In particular,  the ``input-output mapping" $\,\omega \longmapsto \Theta_{x_0}\big(\Gamma_\omega(\cdot) + \omega(\cdot)\big) \equiv  \mathfrak{X} ( \cdot \,, \omega)$ is continuous as a function from the canonical space $\,C([0,\infty); \R)\,$ into the space $\,C \big([0,\infty);  \I \, \big)\,,$ where both spaces are equipped with the topology of uniform convergence on compact  subsets.
\end{enumerate}
\end{thm}

The proof of Theorem~\ref{WZ} is provided at the end of this section.  
On the strength of  Theorems~\ref{WZ} and \ref{T1L}, and under their conditions, the SIE~\eqref{1b}  has a unique solution $X(\cdot)$ on the filtered probability  space $\, ( \Omega, {\cal F}, \Prob),\,  \mathbb{F} =  \{ \mathcal{F} (t) \}_{0 \le t < \infty}$ for any given semimartingale measure $\,\Prob \in \mathfrak{P}$, with $\Prob(\mathcal S(X) = \infty) = 1$.

\begin{cor}[{\bf Wong-Zakai approximations}]  \label{C:WZA}
Under the setting and with the assumptions of Theorem~\ref{WZ}, and for an arbitrary but fixed  semimartingale measure $\,\Prob \in \mathfrak{P}\,$, suppose that $\,\,W_k (\cdot)  =  \int_0^{\, \cdot} \varphi_k (t) \, \dx t\,,~~ k \in \N\,$ are absolutely continuous $\,\Prob$--almost sure approximations of the $\,\Prob$--semimartingale $\, W(\cdot)\,$ in the sense of \eqref{app1}, for some  sequence $\,  \{  \varphi_k (\cdot) \}_{ k \in \N}\,$ of      $\, \mathbb{F}$--progressively measurable and locally integrable  processes.  Let $\,\Gamma(\cdot)\,$ and $\,\{\Gamma_k(\cdot)\}_{k \in \N}\,$ denote the  solutions of the OIE~\eqref{2LL} corresponding to $\, W(\cdot)\,$ and $\,\{W_k(\cdot)\}_{k \in \N}\,$, respectively. 

Then the  processes $\,X_k (\cdot)  \equiv  \Theta_{x_0} (\Gamma_k(\cdot)+ W_k(\cdot))\,, ~~k \in \N \,$  satisfy $\mathbb{P} (\mathcal{S} (X_k) = \infty) =1$ and the analogues of    \eqref{1b}  in the present context, namely, the OIEs 
$$
X_k(\cdot) \,=\, x_0 + \int_0^{\,\cdot} \mathfrak{s} \big(  X_k(t)\big) \Big(   \varphi_k (t)+ \mathfrak{b}\big(t,  W_k(t), X_k( t) \big)   \Big) \, \dx t\,,
$$
and converge almost surely to the solution $\,X(\cdot) = \Theta_{x_0} (\Gamma(\cdot) +  W(\cdot))\,$ of the SIE \eqref{1b}, namely
$$
X(\cdot) \,=\, x_0 + \int_0^{\,\cdot} \mathfrak{s} (  X(t))\, \big[ \,\dx W(t) + \mathfrak{b} (t,W(t),X(t)) \,\dx t \,\big] - \int_{\I} L^X (\cdot\,,   \xi) \,  \mathfrak{s}  (\xi)\, \dx \frac{1 }{ \mathfrak{s}  (\xi)  } \,,
$$
  uniformly over compact intervals, in the manner of \eqref{app2}. 
\end{cor}

We  formulate now a support theorem, which follows almost directly from Corollary~\ref{C:WZA}.  First, we   introduce some necessary notation.   For any given semimartingale measure $\,\Prob \in \mathfrak{P}\,$  and any initial position $\, x_0 \in \I\,$,  we denote by $\,\mathfrak{U}^\Prob (x_0)  \subseteq  C \big([0,\infty);  \I \, \big)\, $ the support  under $\,\Prob\,$ of the solution process $\,X(\cdot)= \Theta_{x_0} (\Gamma_W(\cdot) +  W(\cdot))\,$  for the SIE  \eqref{1b}; this is  the smallest closed subset of   $C \big([0,\infty);  \I \, \big)$,  equipped with  the topology of uniform convergence on compact sets, with the property $\,\Prob \big(X(\cdot) \in \mathfrak{U}^\Prob (x_0) \big) = 1\,$.

Moreover, we let $\,C^{\text{PL}}\, $ and $\,C^\infty \,$ denote, respectively, the spaces of piecewise linear and infinitely differentiable functions $\, \omega : [0,\infty) \rightarrow \R$.
For any given subset $A $ of $\,  C\big([0,\infty);  \I   \big)$, we denote by $\,\overline{A}\,$  its topological closure under  the topology of uniform convergence on compact sets.

\begin{cor}[{\bf Support theorem}] 
Under the setting and with the assumptions of Theorem~\ref{WZ},    and with a  fixed   semimartingale measure $\,\Prob \in \mathfrak{P}\,,$ we have   
\begin{align*}
	\mathfrak{U}^\Prob (x_0)  \subseteq   \overline{\,\big\{ \Theta_{x_0}(\Gamma_\omega(\cdot) + \omega(\cdot))\,:\, \omega \in C^{\text{PL}}   \big\}\,} ~ \quad \text{and} \quad ~
	\mathfrak{U}^\Prob (x_0)   \subseteq   \overline{\,\big\{  \Theta_{x_0}(\Gamma_\omega(\cdot) + \omega(\cdot))  \,:\, \omega \in C^\infty   \big\}\,} .
\end{align*}
Moreover,  if $\,\Prob = \Prob_*\,$ is the Wiener measure,      the above set inclusions    become   equalities.
\end{cor}
\begin{proof}
	The set inclusions follow from Corollary~\ref{C:WZA}  and the fact that both spaces $\,C^{\text{PL}} \,$ and $\,C^\infty \,$ are dense in the space of continuous functions $\, C ([0,\infty); \R)\,$, equipped with the topology of uniform convergence on compact sets. Under the Wiener measure $\Prob_*$, the reverse implications follow from a change-of-measure argument  similar to Lemma~3.1 in \citet{SV_support}.
\end{proof}

\subsection*{The proof of Theorem~\ref{WZ}}

\begin{proof}[Proof of Theorem~\ref{WZ}]
To prove the first part of the theorem, fix a path $\omega (\cdot) \in C([0,\infty); \R)$ and $n \in \N$.  Next, recall the function $G(\cdot\,, \cdot\,, \cdot)\,$ of \eqref{cue} and define
the function   
\begin{align*}
[0, \infty) \times \R \ni (t,\gamma) \longmapsto g_\omega(t,\gamma) = G\left(t, {\omega}(t), \gamma\right) \in \R.
\end{align*}
If $\gamma \in [-n,n]$ we have $ \omega(\cdot) + \gamma \in (\widetilde \ell_m, \widetilde r_m)$ on $[0,n]$ for some sufficiently large $m \in \N$ and thus, the local Lipschitz condition
\begin{align*}
	\big| g_\omega(t,  \gamma_1) - g_\omega(t,\gamma_2)\big| \leq L_m  \big|\Theta_{x_0}(\omega(t) + \gamma_1) - \Theta_{x_0}(\omega(t) + \gamma_2) \big|  \leq   \left( L_m\, \sup_{\xi \in (\ell_m, r_m)} \mathfrak{s}(\xi) \right) |\gamma_1 - \gamma_2|
\end{align*}
for all $t \in [0,n]$ and $(\gamma_1, \gamma_2) \in [-n,n]^2$.  Since the function $g_\omega(\cdot, \cdot)$ is also bounded, 
Carath\'eodory's extension of the Peano  existence theorem \citep[see Theorems~2.1.1 and 2.1.3 in][]{CoddingtonLevinson}  guarantees the existence  of a  solution $\Gamma^{(n)}_\omega(\cdot)$ to the OIE
  $$
  \Gamma^{(n)}_\omega(t) =  x_0 + \int_0^{t} g_\omega\big(t,\Gamma^{(n)}_\omega  (s)\big)\, \dx s\,, \qquad 0 \le t \leq n
  $$
 up to the first time that    $\Gamma^{(n)}_\omega(\cdot)$ leaves the interval $(-n,n)$. 
  Moreover, a Picard-Lindel\"of-type argument yields the uniqueness of the solution, thus  also the non-anticipativity of the function $ (t, \omega) \mapsto \Gamma^{(n)}_\omega( t)$.  
Stitching those solutions for each $n\in \N$ together,   yields then a  unique   non-anticipative    mapping $[0,\infty) \times \Omega \ni (t, \omega) \mapsto \Gamma_\omega( t )$, as in the statement of the theorem. Since $g_\omega(\cdot, \cdot)$ is bounded, we have  $\Gamma_\omega(\cdot) \in \R$.

   In  order to conclude the proof of the first claim, we need to show now that the mapping $\omega \mapsto \Gamma_\omega( \cdot)$ is   measurable; however, while proving  below the second claim of the theorem, we shall show that this mapping is actually continuous, thus, {\it a-fortiori,}   measurable. Now measurability and $\mathbb{F}-$adaptivity -- a consequence of non-anticipativity -- lead to the progressive measurability of this mapping; see Propositions~1.1.12 and 1.1.13 of  \citet{KS1}, in conjunction with the continuity of the mapping $t \mapsto \Gamma_\omega( t )$.

For the second claim of the theorem, let us  fix paths $\{\omega_k (\cdot)\}_{k \in \N}$  and $\omega (\cdot) $ as in the statement, and an integer $n \in \N$.  Let  $\beta< \infty$ denote  an upper bound on  the function $| \mathfrak{b}(\cdot\,, \cdot\,, \cdot)|,$ and fix $m \in \N$ so that  
$$\sup_{\,t \leq n, \, k \in \N} | \omega_k(t) | <  m, \quad x_0+ n \beta + \sup_{\,t \leq n, \, k \in \N} | \omega_k(t)| < \widetilde r_m, \quad \text{and} \quad x_0 - n \beta - \sup_{\,t \leq n, \, k \in \N} | \omega_k(t)| > \widetilde \ell_m.$$
Next, observe that we have
\begin{align*}
	\big|\Gamma_\omega(\cdot) - \Gamma_{\omega_k}(\cdot)\big| 
		&\leq
	\int_0^{n} \big| G(t,  \omega(t), \Gamma_\omega(t)) - G( t, \omega_k(t), \Gamma_\omega(t) ) \big| \, \dx t  \\  &+ \int_0^\cdot \big| G( t, \omega_k(t), \Gamma_\omega(t))  -  G( t, \omega_k(t), \Gamma_{\omega_k}(t))\big| \dx t 
\end{align*}
on $[0,n]$ for all $k \in \N$. Since the function $G(\cdot\,, \cdot\,, \cdot)$ is bounded, the Dominated Convergence Theorem yields that the first term on the right-hand can be made arbitrarily small. For the second term, we can use the Lipschitz continuity of $G(\cdot\,, \cdot\,, \cdot)$ with Lipschitz constant $L_m$. An application of Gronwall's lemma then yields that $\lim_{k \uparrow \infty} \sup_{t \leq n} \big|\Gamma_\omega(t) - \Gamma_{\omega_k}(t)\big| = 0$.
Since the function $\Theta_{x_0}(\cdot)$ is locally Lipschitz continuous, the statement follows.
\end{proof}

\appendix
\section{Appendix: Regularization of OIEs}
 \label{S: regul}

The implications of  Theorem~\ref{T1L} or Corollary~\ref{T1} can prove useful for obtaining existence and uniqueness statements of OFEs in the form of the OFE~\eqref{7L}. For instance, 
Theorem~\ref{T:5.2} is a case in point.

\begin{example}  
Let us look closer at the setup of the SIE~\eqref{1c}   
under the Wiener measure $\,\Prob_*\,$ and with $\mathfrak{s}(\cdot) \equiv 1$, for a   bounded, measurable function $\, \mathfrak{b}   : [0, \infty) \times  \R \ra \R\,$. It is well known (see, for example, \citet{Zvonkin1974} or \citet{Veretennikov1981}) that   the  resulting SIE
\begin{equation}
\label{ZV}
X(\cdot)\,=\, x_0 + \int_0^{\, \cdot} \mathfrak{b} \big( t, X(t) \big) \, \mathrm{d} t + W(\cdot) 
\end{equation}
\noindent
 has a   unique  $\,\mathbb{F}
$--adapted and non-exploding solution.  In fact, \citet{Krylov2005} show that the SIE~\eqref{ZV}  admits a pathwise unique,  strong and non-explosive solution,  under  only very weak 
integrability conditions on the function $\mathfrak{b}(\cdot\,, \cdot)$;  see also \citet{Fedrizzi2011} for a simpler argument. Theorem~\ref{T1L} now implies that the corresponding OIE~\eqref{2LL}, now in the form
\begin{align}  \label{eq:Davie OIE}
\Gamma(\cdot) =  x_0 +  \int_0^{\, \cdot} \mathfrak{b} \big( t, \Gamma(t) +W(t) \big)\,  \mathrm{d} t \,,
\end{align}
also has a unique $\,\mathbb{F}
$--adapted solution $\Gamma(\cdot)$; a similar point is made by  \citet{Davie}.

We do not know a theory of OIEs that can prove such existence and uniqueness statements of this type. An explanation is given in Section~1.6 of \citet{Flandoli2013}:  {\it ``...the intuitive reason behind these uniqueness results in spite of the
singularities of the drift [is]...the regularity of the occupation measure. The measure distributed by single trajectories of diffusions...is ...very regular and diffused with respect to the occupation measure of solutions to deterministic ODEs. This regularity smooths out the singularities of the drift, opposite to the deterministic case in which the solution may persist on the singularities.'' }

The question    answered affirmatively by \citet{Davie} (see also \citet{Flandoli2011} for a simpler argument), is whether uniqueness holds for the OIE~\eqref{eq:Davie OIE} also   for almost all realizations of the Brownian paths $W(\omega)$, among all (possibly   not $\,\mathbb{F}$--adapted) functions $\Gamma(\cdot)$.   For a discussion of the subtle differences in those notions of uniqueness, we refer to the comments after Definition~1.5 in  \citet{Flandoli2011_LN}.  
Recently,  \citet{Catellier2014} further extended the regularization results to paths of fractional Brownian motion.  \qed
 \end{example}

\begin{example}  
In the context of the SIE~\eqref{eq:1.3b} in Subsection~\ref{sec3.1}, under the Wiener measure $\,\Prob_*\,$  and with $\,\mathfrak{s}(\cdot) \equiv 1\,$ and $\,\mathfrak{b}(\cdot) \equiv 0\,$, the OFE~\eqref{77c} can be simplified to   
\begin{equation*}
\Gamma(\cdot)=  x_0 +   \int_\I L^{ \Gamma +  W}(\cdot \,, \xi  ) \, \bm \mu(\dx \xi)
\end{equation*}
up to an explosion time; if additionally $\, \I = \R\,$ and $\, \bm \mu(\dx \xi) = \beta \,\bm \delta_0(\dx \xi)$ for some $\beta \in \R\,$ and the Dirac measure $\bm \delta_0(\cdot)$ at the origin,   this expression simplifies further to the equation
\begin{align}  
\label{eq:OIE Skew}
\Gamma(\cdot) = x_0 + \beta \, L^{ \Gamma +  W}(\cdot\,, 0 ) .
\end{align}

Not much can be said directly about this equation;  we note, however, that the corresponding stochastic equation~\eqref{eq:1.3b} simplifies to 
\begin{equation}
\label{skew}
X(\cdot) \,=\, x_0 + W(\cdot) + \beta \, L^X (\cdot\,, 0)\,,
\end{equation}
the stochastic equation for the Skew Brownian Motion with skewness parameter $\, \alpha = 1/(2-\beta)$.  In terms of the symmetric local time at the origin $\,
\widehat{L}^X (\cdot\,, 0)  =  ( 1 / 2)  \left( L^{X} (\cdot\,, 0) + L^{-X} (\cdot\,, 0) \right)\,$, the above equation can be written in the equivalent and perhaps more ``canonical" form 
$$
X(\cdot) \,=\, x_0 + W(\cdot) + \gamma \, \widehat{L}^X (\cdot\,, 0) \qquad \text{
with}\quad \, \gamma \,=\, \frac{2\,\beta}{\,2-\beta\,}\, =\, 2 ( 2  \alpha -1)\,.$$ 

In accordance with Theorem~\ref{T:5.2}, the SIE~\eqref{skew} has a pathwise unique, strong solution for all $\, \beta< 1\,$; the theory of the Skorokhod reflection problem guarantees such a solution for $\, \beta=1\,$; whereas it is shown in \citet{HarrisonShepp} that  there is no such solution for $\, \beta >1\,$.  From Theorem~\ref{T1L}, analogous  statements hold  for $\mathbb{F}$--adapted solutions to the OFE~\eqref{eq:OIE Skew}.  \qed
\end{example}

\bibliography{aa_bib}{}
\bibliographystyle{apalike}

\end{document}